\documentclass{article}
\usepackage[margin=1.3in]{geometry}
\usepackage{amsmath,amssymb,amsthm,mathrsfs,color,times,textcomp,sectsty,verbatim}
\usepackage{xcolor}
\usepackage[colorlinks=true]{hyperref}
\hypersetup{urlcolor=blue, citecolor=red, linkcolor=blue}
\usepackage[toc,page]{appendix}
\numberwithin{equation}{section}
\theoremstyle{plain}
\newtheorem{theorem}{Theorem}[section]
\newtheorem{proposition}[theorem]{Proposition}
\newtheorem{lemma}[theorem]{Lemma}

\newtheorem{corollary}[theorem]{Corollary}

\theoremstyle{definition}

\newcommand\backmatter{\appendix
\def\chaptermark##1{\markboth{%
\ifnum  \c@secnumdepth > \m@ne  \@chapapp\ \thechapter:  \fi  ##1}{%
\ifnum  \c@secnumdepth > \m@ne  \@chapapp\ \thechapter:  \fi  ##1}}%
\def\sectionmark##1{\relax}}

\def \no{\nonumber}
\def \pa{\partial}

\def\e{\epsilon}

\def\R{\mathbb{R}}
\def\Rn{{\mathbb{R}}^n_+}
\def\d{\partial}

\def\D{\Delta}

\def\crit {\frac{2n}{n-2}}

\def\conj {\frac{2n}{n+2}}

\def\ba{\begin{align}}
\def\ea{\end{align}}
\def\bp{\begin{proof}}
\def\ep{\end{proof}}

\def\func:u{\bar{u}_{(x_0,\e)}}

\def\ks{(\xi,\epsilon)}
\def\ksi{(\xi,\epsilon,i)}
\def\ksn{(\xi,\epsilon,n)}
\def\ksa{(\xi,\epsilon,a)}

\newcommand{\ud}{\mathrm{d}}
\newcommand{\Sn}{\mathbb{S}^n}
\newcommand{\Sp}{\mathbb{S}}
 \allowdisplaybreaks

\begin{document}

\title{\Large \bf
Blow-up phenomena for the constant scalar curvature and constant boundary mean curvature equation}
\author{Xuezhang  Chen\thanks{X. Chen is partially supported by NSFC (No.11771204), A Foundation for the Author of National Excellent Doctoral Dissertation of China (No.201417) and start-up grant of 2016 Deng Feng program B at Nanjing University. Email: xuezhangchen@nju.edu.cn.} ~ and Nan Wu\thanks{N. Wu: 141110012@smail.nju.edu.cn.}\\
 \small
$^\ast$$^\dag$Department of Mathematics \& IMS, Nanjing University, Nanjing
210093, P. R. China
}

\date{}

\maketitle

\begin{abstract}
 We first present a warped product manifold with boundary to show the non-uniqueness of the positive constant scalar curvature and positive constant boundary mean curvature equation. Next, we construct a smooth counterexample to show that the compactness of the set of ``lower energy" solutions to the above equation fails when the dimension of the manifold is not less than $62$.

{{\bf $\mathbf{2010}$ MSC:} Primary 53C21, 35J20; Secondary 35B33, 34B18.}

\end{abstract}

%\listoftodos

\section{Introduction}

Let $(M^n, g)$ be a smooth compact Riemannian manifold with boundary $\d M$ and $n\geq 3$. In 1999, Zheng-Chao Han and Yan Yan Li \cite{han-li2} proposed a question of finding a conformal metric with positive constant scalar curvature and any constant boundary mean curvature in the positive Yamabe constant conformal class. In analytical terms, it corresponds to the existence of a positive solution to
\begin{align}\label{PDE:scalar_mean_curvatures}
\begin{cases}
\displaystyle -\D_{g}u+c_n R_g u=c_1 u^{\frac{n+2}{n-2}}, &\quad \mathrm{in~~} M, \\
\displaystyle \frac{\d u}{\d \nu_g}+d_n h_g u=c u^{\frac{n}{n-2}}, &\quad \mathrm{on~~} \d M,
\end{cases}
\end{align} 
for $c_1\in \mathbb{R}_+,c \in \mathbb{R}$, where $c_n=(n-2)/[4(n-1)]$ and $d_n=(n-2)/2$, $R_g$ is the scalar curvature and $h_g $ is the mean curvature of $\d M$ and $\nu_g$ is the outward unit normal on $\d M$.  This existence problem has been studied by Z. C. Han and Y. Y. Li \cite{han-li1,han-li2}, and recently by the first named author and his collaborators \cite{Chen-Sun,ChenRuanSun}. The closely related works are referred to  J. Escobar \cite{escobar5,escobar3} etc. X. Chen, Y. Ruan an L. Sun \cite{ChenRuanSun} introduced a ``free" functional
\begin{align}\label{CRS_functional_I}
I_{(M,g)}[u]=&\int_{M}\left(\frac{4(n-1)}{n-2}|\nabla u|_{g}^2+R_{g}u^2\right)\ud\mu_{g}+2(n-1)\int_{\partial M} h_{g}u^2\ud\sigma_{g}\no\\
&-\frac{4(n-1)}{n}c_1\int_{M}u_+^{\frac{2n}{n-2}}\ud\mu_{g}-4c\int_{\partial M}u_+^{\frac{2(n-1)}{n-2}}\ud\sigma_{g}
\end{align}
for any $u \in H^1(M,g)$, where $u_+=\max\{u,0\}$. The authors applied the Mountain Pass Lemma to show the existence of PDE \eqref{PDE:scalar_mean_curvatures} for all $c \in \mathbb{R}$, in addition that the energy of solutions below a threshold $S_c$, except for the case that $n\geq 8$, $\pa M$ is umbilic, the Weyl tensor of $M$ vanishes on $\pa M$ and has an interior non-zero point. Here the geometric meaning of $S_c$ is the energy $I_{(\Rn,|\ud x|^2)}$ of a single bubble $u_{(0,1)}(x)$; see Section \ref{Sect3} or  \cite{ChenRuanSun}. We will present an example in Section \ref{Sect2} to show the non-uniqueness of PDE \eqref{PDE:scalar_mean_curvatures} when $c_1, c\in \mathbb{R}_+$. Indeed, Han-Li \cite{han-li2} established the compactness of the full set of positive solutions to PDE \eqref{PDE:scalar_mean_curvatures} for all $c \leq \bar c$ with any given positive constant $\bar c$  (see \cite[Conjecture 2]{han-li1} and \cite[Theorem 0.1]{han-li2}), provided that $M$ is locally conformally flat with umbilic boundary, and is not conformally equivalent to the standard hemisphere $\Sn_+$. 

Denote by $L_g=-\Delta_{g}+c_nR_{g}$ the conformal Laplacian and $B_{g}=\frac{\partial}{\partial \nu_g}+d_nh_g$ the boundary conformally covariant operator, respectively. Both $L_g$ and $B_g$ have the following conformally covariant properties: Let $\tilde g=u^{4/(n-2)}g$, then for any $\varphi \in C^\infty(\overline M)$, there hold
\begin{equation}\label{eq:conformal_invariance}
L_{g}(u\varphi)=u^{\frac{n+2}{n-2}}L_{\tilde g}(\varphi) \mathrm{~~and~~} B_{g}(u\varphi)=u^{\frac{n}{n-2}}B_{\tilde g}(\varphi).
\end{equation}
The Yamabe constant is defined by
\begin{align*}
Y(M,\d M, [g]) = \underset{\tilde g\in [g] }{\inf} \frac{\int_M R_{\tilde g} \ud\mu_{\tilde g} +2(n-1)\int_{\d M}h_{\tilde g} \ud\sigma_{\tilde g}}{(\int_M \ud\mu_{\tilde g})^{\frac{n-2}{n}}}.
\end{align*}

For the closed manifolds, the question of compactness of the full set of solutions to the Yamabe equation was initiated by R. Schoen in $1988$. A necessary condition is that the manifold is not conformally equivalent to the standard sphere $\Sn$. It has been extensively studied by R. Schoen \cite{Schoen1,Schoen2}, Y.Y. Li and M. Zhu \cite{li-zhu2}, O. Druet \cite{Druet}, F. Marques \cite{marques2}, Y.Y. Li and L. Zhang \cite{li-zhang,li-zhang2}, etc. Eventually, the compactness for dimensions $3 \leq n\leq 24$ with assuming positive mass theorem was established by M. Khuri, F. Marques and R. Schoen \cite{khuri1}. For the non-compactness part, S. Brendle \cite{Brendle3} discovered the first smooth counterexamples in dimensions $n\geq 52$. S. Brendle and F. Marques \cite{Brendle-Marques} extended the above counterexample to the remaining dimensions $25 \leq n \leq  51$.

For the manifolds with boundary,  the blow-up phenomena in dimensions $n\geq 25$ were discovered by Almaraz \cite{almaraz1} corresponding to $c_1=0, c>0$ in \eqref{PDE:scalar_mean_curvatures}. Such blow-up phenomena in large dimensions also appear in the $Q$-curvature equation (see J. Wei and C. Zhao \cite{Wei-Zhao} ) and in the fractional Yamabe problem (see S. Kim, M. Musso and J. Wei \cite{Kim-Musso-Wei}).

It is natural to expect that the blow-up phenomena of PDE \eqref{PDE:scalar_mean_curvatures} occur in large dimensions. Now we confirm it in dimensions $n \geq 62$ for the set of solutions whose energy of $I_{(M,g)}$ is below $S_c$.
\begin{theorem}\label{Thm:main}
 For $n\geq 62$, there exists a smooth Riemannian metric $g$ on $\Sn_{+}$, such that $Y(\Sn_+,\d \Sn_+, [g]) > 0$ and a sequence of positive smooth functions $\{v_\nu; \nu \in \mathbb{N}\}$ with the following properties:
 \begin{enumerate}
 \item[(i)] $g$ is not conformally flat;
\item[(ii)] $\d \Sn_+$ is umbilic with respect to $g$;
\item[(iii)] for all $\nu$,\,$v_\nu$ is a positive solution to PDE \eqref{PDE:scalar_mean_curvatures} with $c_1,c>0$ and $M=\Sn_{+}$;
 \item[(iv)] $I_{(\Sn_+,g)}[v_\nu]<S_c$;
 \item[(v)] $\sup_{\d \Sn_{+}}v_\nu \to \infty$ as $\nu \to \infty$.
 \end{enumerate}
\end{theorem}

Since the compactness results of PDE \eqref{PDE:scalar_mean_curvatures} are not abundant yet, the critical dimension of the non-compactness is not a main issue in this paper and thus left to future study.

The paper is arranged as follows. We show the multiplicity of PDE \eqref{PDE:scalar_mean_curvatures} on a warped product manifold with boundary, which is presented in Section \ref{Sect2}. In Section \ref{Sect3}, we describe how the problem can be reduced to finding critical points of a certain function $\mathcal{F}_g\ks$, where $\xi$ is a vector in $\R^{n-1}$ and $\epsilon$ is a positive real number. In Section \ref{Sect4}, we show that the function $\mathcal{F}_g\ks$ can be approximated by an auxiliary function $F\ks$. In Section \ref{Sect5}, we prove that the function $\mathcal{F}\ks$ has a strict local minimum point $(0,1)$. Finally, in Section \ref{Sect6}, we use a perturbation argument to find critical points of $\mathcal{F}_g\ks$ and then show the non-compactness.

\vskip .2 in
\noindent \textbf{Acknowledgement:} Part of this work was carried out while both authors were visiting the Department of Mathematics at Rutgers University, to which  they are grateful for providing the very stimulating research environment and supports.
 
\section{Non-uniqueness: an example}\label{Sect2}
The purpose of this section is to construct a warped product manifold with boundary, which demonstrates the multiplicity of solutions of \eqref{PDE:scalar_mean_curvatures}. This is somewhat inspired by the one for the Yamabe problem in \cite[p.178]{aubin_book}.
\begin{proposition}
If $n\geq 5$ and $c_1,c_2>0$, then PDE \eqref{PDE:scalar_mean_curvatures} admits at least two positive smooth solutions.
\end{proposition}
\begin{proof}
Suppose $(M^{n_1} , g_1)$ is an $n_1$-dimensional ($n_1\geq 3$) smooth compact manifold with boundary such that $R_{g_{1}}$ and $h_{g_1}$ are two positive constants, which is guaranteed by \cite[Theorem 1.1]{ChenRuanSun} for almost all smooth manifolds with boundary. Let $(M^{n_2}, g_2)$ be an $n_2$-dimensional ($n_2\geq 2$) smooth closed manifold with positive constant scalar curvature $R_{g_2}$. Consider a warped product manifold $(M^n,g)=(M^{n_1}\times M^{n_2}, k g_1+g_2)$, where $n=n_1+n_2\geq 5$ and $k$ is a positive constant.
Obviously, $\d M=\d M^{n_1} \times M^{n_2}$ and the second fundamental form $\pi$ on $\d M$ satisfies:
\begin{align*}
\pi_{ij}=k^{\frac{1}{2}}\pi^{g_1}_{ij}, \pi_{iI}=0, \pi_{IJ}=0 \quad 1\leq i,j \leq n_1-1, 1\leq I,J\leq n_2,
\end{align*}
which means $h_g=k^{-\frac{1}{2}}h_{g_1}$ and $R_g=k^{-1}R_{g_1}+R_{g_2}$.

Next we claim that the following PDE
\begin{align}\label{prob:multiplicity}
\begin{cases}
\displaystyle -\D_{g}u+ c_n R_g u= c_n R_g u^{\frac{n+2}{n-2}}, &\quad \mathrm{~~in~~} M, \\
\displaystyle \frac{\d u}{\d \nu_g}+d_n h_g u=d_n h_g u^{\frac{n}{n-2}}, &\quad \mathrm{~~on~~} \d M,
\end{cases}
\end{align}
has at least two positive smooth solutions.

To that end, first notice that $1$ is a solution of \eqref{prob:multiplicity}. On the other hand, it follows from \cite[Theorem 1.1]{ChenRuanSun} that there exists a positive smooth mountain critical point  $u_0$ of $I_{(M,g)}$, which is the one in \eqref{CRS_functional_I} with $c_1=c_nR_g$ and $c=d_nh_g$, such that
\begin{align*}
I_{(M,g)}[u_0]<S_c(k),
\end{align*}
where we use $S_c(k)$ instead of $S_c$ to emphasize $k$. If we replace $(M,g)$ by $(\Rn,|\ud x|^2)$ in \eqref{prob:multiplicity} with its positive solution denoted by 
$$W(x)=\left(\frac{n(n-1)}{R_g}\right)^{\frac{n-2}{4}}\left(\frac{2}{1+|x-T_c\textbf{e}_n|^2}\right)^{\frac{n-2}{2}}\quad \mathrm{with~~} T_c=-\frac{h_g}{2},$$
which is also called a standard bubble. Furthermore, as $k \to \infty$, by the dominated convergence theorem we have
\begin{align*}
S_c(k)=&I_{(\Rn,g_{\mathbb{R}^n)}} [W]\\
=&\frac{4(n-1)}{n-2}\int_{\Rn}|\nabla W|^2 \ud x-\frac{n-2}{n}R_g\int_{\Rn} W^{\frac{2n}{n-2}}\ud x-2(n-2)h_g \int_{\pa \Rn}W^{\frac{2(n-1)}{n-2}} \ud \sigma\\
=&\frac{2}{n} R_g\int_{\Rn} W^{\frac{2n}{n-2}}\ud x+2h_g \int_{\pa \Rn}W^{\frac{2(n-1)}{n-2}} \ud \sigma\\
\to& \frac{2}{n}R_{g_2} \left(\frac{n(n-1)}{R_{g_2}}\right)^{\frac{n}{2}}\int_{\Rn} \left(\frac{2}{1+|x|^2}\right)^n \ud x\\
=&\frac{2}{n}R_{g_2} \left(\frac{n(n-1)}{R_{g_2}}\right)^{\frac{n}{2}}\frac{\omega_{n}}{2}=R_{g_2} \left(\frac{n(n-1)}{R_{g_2}}\right)^{\frac{n}{2}}\frac{\omega_n}{n}:=S_c(\infty).
\end{align*}
For simplicity, we let 
$$V_1=\mathrm{Vol}(M^{n_1},g_1), ~~\hat V_1=\mathrm{Vol}(\pa M^{n_1},g_1),~~V_2=\mathrm{Vol}(M^{n_2},g_2).$$
Indeed, if $k$ is large enough, then $u_0$ is distinct from $1$. This follows from
\begin{align*}
I_{(M,g)}[1]=&\int_{M} R_g \ud\mu_g +2(n-1)\int_{\d M}h_g \ud\sigma_g -\frac{n-2}{n}R_g \mathrm{Vol}(M)-2(n-2)h_g \mathrm{Vol}(\d M)\\
=&\frac{2}{n} (k^{-1}R_{g_1}+R_{g_2})k^{\frac{n_1}{2}}V_1 V_2+2k^{-\frac{1}{2}}h_{g_1} k^{\frac{n_1-1}{2}}\hat{V_1}V_2\\
>&S_c(\infty)+1>S_c(k)>I_{(M,g)}[u_0],
\end{align*}
if $k$ is sufficiently large. 
\end{proof}

\section{Lyapunov-Schmidt reduction}\label{Sect3}
From now on, let $c_1=n(n-2)$ and $T_c=-c/(n-2)$ be a negative real number for brevity. Given a pair $\ks\in\R^{n-1}\times (0,\infty)$ we define
 $$ u_{\ks}(x)=\left(\frac{\epsilon}{\epsilon^2+(x_n-T_c \epsilon)^2+|x'-\xi|^2} \right)^\frac{n-2}{2},$$
 where $x=(x',x_n) \in \Rn$.
 Then $u_{\ks}$ satisfies
 \begin{align}\label{bubble_half_space}
 \begin{cases}
 \displaystyle -\D u_{\ks}=n(n-2)u_{\ks}^{\frac{n+2}{n-2}}, &\quad \mathrm{in~~} \Rn, \\
\displaystyle \frac{\d u_{\ks}}{\d x_n}=(n-2)T_c u_{\ks}^{\frac{n}{n-2}}, &\quad \mathrm{on~~} \d \Rn.
\end{cases}
\end{align}
This implies that the metric $u_{\ks}^{4/(n-2)}|\ud x|^2$ is Einstein, then there holds
 \begin{equation}\label{eq:Einstein}
 \d_a u_{\ks} \d_b u_{\ks}-c_n\d_a\d_b u_{\ks}^{2}=\frac{1}{n}\left(|\nabla u_{\ks}|^2 -c_n\Delta u_{\ks}^{2} \right)\delta_{ab}
 \end{equation}
for $1 \leq a,b \leq n$. Define 
 \begin{align*}
  u_{\ksi}(x)=&\left(\frac{\epsilon}{\epsilon^2 +(T_c \epsilon-x_n)^2+|x'-\xi|^2}\right)^\frac{n+2}{2}\frac{2\epsilon(x_i -\xi_i)}{\epsilon^2 +(T_c \epsilon-x_n)^2+|x'-\xi|^2},\\
 \hat{u}_{\ksi}(x)=&\left(\frac{\epsilon}{\epsilon^2 +(T_c \epsilon-x_n)^2+|x'-\xi|^2}\right)^\frac{n}{2}\frac{2\epsilon(x_i -\xi_i)}{\epsilon^2 +(T_c \epsilon-x_n)^2+|x'-\xi|^2}
 \end{align*}
for $1\leq i\leq n-1$,
 and
\begin{align*}
u_{\ksn}(x)=&\left(\frac{\epsilon}{\epsilon^2 +(T_c \epsilon-x_n)^2+|x'-\xi|^2}\right)^\frac{n+2}{2}\frac{(1+T_{c}^{2})\epsilon^2-x_{n}^{2}-|x'-\xi|^2}{\epsilon^2 +(T_c \epsilon-x_n)^2+|x'-\xi|^2},\\
 \hat{u}_{\ksn}(x)=&\left(\frac{\epsilon}{\epsilon^2 +(T_c \epsilon-x_n)^2+|x'-\xi|^2}\right)^\frac{n}{2}\frac{(1+T_{c}^{2})\epsilon^2-x_{n}^{2}-|x'-\xi|^2}{\epsilon^2 +(T_c \epsilon-x_n)^2+|x'-\xi|^2}.
 \end{align*}
Obviously, $\|{u_{\ksa}}\|_{L^\conj(\Rn)}$ and $\|{u_{\ksa}}\|_{L^{\frac{2(n-1)}{n}}(\d\Rn)}$ are constant in $\xi$ and $\epsilon$ for $1\leq a\leq n$.

Define
$$\mathcal{E}=\left\{w\in L^{\crit}(\Rn)\cap L^{\frac{2(n-1)}{n-2}}(\d\Rn)\cap H^{1}_{\mathrm{loc}}(\Rn); \int_{\Rn} |\nabla w|^2 <\infty \right\}$$
and
\begin{align*}
\mathcal{E}_{\ks}=\left\{w\in \mathcal{E}; 2n\int_{\Rn} w u_{\ksi}-T_c\int_{\d\Rn}w\hat{u}_{\ksi} =0~~\mathrm{for~~all~~}1\leq i\leq n\right\}.
\end{align*}
We define a norm on $\mathcal{E}$ by $\|w\|_{\mathcal{E}}=\left( \int_{\Rn}|\nabla w|^2 \right)^{1/2}$. Clearly, $u_{\ks} \in \mathcal{E}_{\ks}$. 

It follows from \cite[Theorem 3.3]{escobar5} that there exists an optimal constant $K=K(n)>0$ such that
\begin{equation}\label{ineq: Sobolev and trace}
\left( \int_{\Rn} w^{\frac{2n}{n-2}} \right)^{\frac{n-2}{n}}+\left(  \int_{\d \Rn} w^{\frac{2(n-1)}{n-2}}\right)^{\frac{n-2}{n-1}}\leq K\int_{\Rn}|\nabla w|^2
\end{equation}
for all $w\in \mathcal{E}$.

Let $\pi:\eta \in \Sn \mapsto x \in \Rn$ be the stereographic projection (see \cite[Fig.1 on p.9]{Chen-Sun}), which is given by
\begin{align*}
\begin{cases}
\displaystyle \eta_{i}=\frac{2x_i}{1 +(T_c -x_n)^2+|x'|^2},\quad\mathrm{for~~}1\leq i\leq n-1,\\
\displaystyle \eta_{n}=\frac{2(x_n -T_c)}{1 +(T_c-x_n)^2+|x'|^2},\\
\displaystyle \eta_{n+1}=\frac{|x'|^2 +(x_n -T_c )^2 -1}{1 +(T_c-x_n)^2+|x'|^2}.
\end{cases}
\end{align*}
Let $\Sigma=\pi^{-1}(\Rn)$ be a spherical cap equipped with the standard round metric $g_\Sigma=g_{\Sn}$. If we choose the center of $\Sigma$ as the north pole, the coordinate system $\xi$ is changed to another coordinate system $\zeta$ by
\begin{align*}
\begin{cases}
\displaystyle \zeta_i=\eta_i, \qquad \mathrm{for~~}1\leq i\leq n-1,\\
\displaystyle \zeta_n=-\frac{T_c\eta_n+\eta_{n+1}}{\sqrt{1+T_c^2}},\\
\displaystyle \zeta_{n+1}=\frac{\eta_n-T_c\eta_{n+1}}{\sqrt{1+T_c^2}}.
\end{cases}
\end{align*}
\begin{proposition}\label{prop:eigenvalues}
There exist a positive constant $\mu$ depending only on $n$ and $T_c $, such that
\begin{align}\label{eigen esti}
& \frac{1}{2}\left[\int_{\Rn}\left(|\nabla w|^2 -n(n+2)u_{\ks}^{\frac{4}{n-2}}w^2\right) +n T_c \int_{\d \Rn}u_{\ks}^{\frac{2}{n-2}}w^2\right]\no \\
&+\frac{(1+(n-2)\mu)2^{n-1}}{n|\Sigma|-T_c|\d\Sigma|}\left[ 2n\int_{\Rn}u_{\ks}^{\frac{n+2}{n-2}}w-T_c \int_{\d \Rn}u_{\ks}^{\frac{n}{n-2}}w\right]^2\geq \mu \|w\|_{\mathcal{E}}^2
\end{align}
for all $w \in \mathcal{E}_{\ks}$.
\end{proposition}
 \begin{proof}
Given a function $\Phi$ defined on $\Sigma$, we define 
 \begin{align*}
 \phi(x)=2^{\frac{n-2}{2}}\Phi\circ \pi^{-1}(x)u_{\ks}(x), \quad x \in \mathbb{R}_+^n.
 \end{align*}
It follows from \cite[Proposition 3.4]{han-li1} that there exists a positive constant $\mu$ depending on $n$ and $T_c$, such that
\begin{align}\label{est:Han-Li_eigenvalue}
Q_2(\Phi_0):=&\frac{1}{2}\left[\int_{\Sigma} (|\nabla\Phi_0|_{g_\Sigma}^2 -n\Phi_0^2)+T_c\int_{\d\Sigma}\Phi_0^2\right]\no\\
&+\frac{1}{2(n|\Sigma|-T_c|\d\Sigma|)}\left(n\int_{\Sigma}\Phi_0-T_c\int_{\d\Sigma}\Phi_0 \right)^2 \no\\
\geq&\mu\left[\int_{\Sigma}|\nabla\Phi_0|_{g_\Sigma}^2+\frac{n(n-2)}{4}\Phi_0^2-\frac{n-2}{2}T_c\int_{\d\Sigma}\Phi_0^2 \right]
\end{align}
for all $\Phi_0\in \mathrm{span}\{1,\zeta_1,\zeta_2,\cdots,\zeta_n\}^{\bot}$, which denotes the orthogonal complement of $\mathrm{span}\{1,\zeta_1,\zeta_2,\cdots,\zeta_n\}$ in $H^1(\Sigma)$.

For any $\Phi \in \mathrm{span}\{\zeta_1,\zeta_2,\cdots,\zeta_n\}^{\bot}$, we set $\Phi=\Phi_0+\bar \Phi$, where $\bar\Phi=\frac{1}{|\Sigma|}\int_{\Sigma}\Phi$. Then by \eqref{est:Han-Li_eigenvalue} we have
\begin{align*}
&\frac{1}{2}\left[\int_{\Sigma} (|\nabla\Phi|_{g_\Sigma}^2 -n\Phi^2)+T_c\int_{\d\Sigma}\Phi^2\right]+\frac{1+(n-2)\mu}{2(n|\Sigma|-T_c|\d\Sigma|)}\left(n\int_{\Sigma}\Phi-T_c\int_{\d\Sigma}\Phi \right)^2 \\
&-\mu\left[\int_{\Sigma}\left(|\nabla\Phi|_{g_\Sigma}^2+\frac{n(n-2)}{4}\Phi^2\right)-\frac{n-2}{2}T_c\int_{\d\Sigma}\Phi^2 \right]\\
=&\frac{1}{2}\left[\int_{\Sigma} (|\nabla\Phi_0|_{g_\Sigma}^2 -n\Phi_{0}^{2})+T_c\int_{\d\Sigma}\Phi_{0}^{2}\right]-\frac{1}{2}(n|\Sigma|-T_c|\d\Sigma|)\bar\Phi^2+T_c \bar \Phi\int_{\d\Sigma}\Phi_0\\
&+\frac{1+(n-2)\mu}{2(n|\Sigma|-T_c|\d\Sigma|)}\left[(n|\Sigma|-T_c|\d\Sigma|)\bar \Phi-T_c\int_{\d\Sigma}\Phi_0 \right]^2 \\
&-\mu\left[\int_{\Sigma}|\nabla\Phi_0|_{g_\Sigma}^2+\frac{n(n-2)}{4}\Phi_{0}^{2}-\frac{n-2}{2}T_c\int_{\d\Sigma}\Phi_{0}^{2}\right]\\
&-\left[-(n-2)T_c\bar \Phi\int_{\pa \Sigma}\Phi_0+\frac{n(n-2)}{4}\bar \Phi^2|\Sigma|-\frac{n-2}{2}T_c\bar \Phi^2|\pa \Sigma|\right]\\
=&\frac{1}{2}\left[\int_{\Sigma} (|\nabla\Phi_0|_{g_\Sigma}^2 -n\Phi_{0}^{2})+T_c\int_{\d\Sigma}\Phi_{0}^{2}\right]+\frac{1+(n-2)\mu}{2(n|\Sigma|-T_c|\d\Sigma|)}\left(T_c\int_{\d\Sigma}\Phi_0 \right)^2\\
&-\mu\left[\int_{\Sigma}|\nabla\Phi_0|_{g_\Sigma}^2+\frac{n(n-2)}{4}\Phi_{0}^{2}-\frac{n-2}{2}T_c\int_{\d\Sigma}\Phi_{0}^{2}\right]+\frac{n(n-2)}{4}\mu \bar \Phi^2 |\Sigma| \geq 0.
\end{align*}
By \eqref{eq:conformal_invariance} we have
\begin{align*}
\int_{\Rn}|\nabla\phi|^2 =&-\int_{\Rn}\phi\Delta\phi +\int_{\d\Rn}\frac{\d\phi}{\d\nu}\phi\\
=&\int_{\Rn}\phi L_{|\ud x|^2}(\phi)+\int_{\d\Rn}\phi B_{|\ud x|^2}(\phi)\\
=&\int_{\Sigma}|\nabla\Phi|_{g_\Sigma}^2 +\frac{n(n-2)}{4}\Phi^2 -\frac{n-2}{2}T_c\int_{\d\Sigma}\Phi^2
\end{align*}
and 
\begin{align*}
&\int_{\Rn}u_{\ks}^{\frac{4}{n-2}}\phi^2 =\frac{1}{4}\int_{\Sigma}\Phi^2, \quad \int_{\d\Rn}u_{\ks}^{\frac{2}{n-2}}\phi^2=\frac{1}{2}\int_{\d\Sigma}\Phi^2,\\
&\int_{\Rn}u_{\ks}^{\frac{n+2}{n-2}}\phi =2^{-\frac{n+2}{2}}\int_{\Sigma}\Phi, \quad \int_{\d\Rn}u_{\ks}^{\frac{n}{n-2}}\phi =2^{-\frac{n}{2}}\int_{\d\Sigma}\Phi.
\end{align*}
Therefore, we combine these facts together to obtain the desired estimate.
 \end{proof}

\begin{proposition}\label{prop:conf_operators}
Consider a Riemannian metric in $\overline{\R_+^n}$ of the form $g(x) = \exp(h(x))$, where $h(x)$ is a trace-free symmetric two-tensor in $\overline{\R_+^n}$ satisfying $|h(x)| + |\d h(x)| + |{\d}^2h(x)|\leq\alpha\leq 1$, and $h_{na}(x)=0$ for all $x\in \overline{\Rn}$ and $h(x)=0$ for all $x \in \overline{\Rn}\setminus B_1^+(0)$. Then there exists a constant $C$, depending only on $n$ and $T_c$, such that
\begin{align*}
&\|\Delta_g u_{\ks} -c_n R_g u_{\ks} + n(n-2)u_{\ks}^{\frac{n+2}{n-2}}\|_{L^{\conj}(\Rn)}+\| h_g u_{\ks}\|_{L^{\frac{2(n-1)}{n}}(\d \Rn)}\leq C\alpha.
\end{align*}
\end{proposition}

\begin{proof}
It is not hard to verify that $h_g=O(|h|)$, which together with  \cite[Proposition 2.3]{almaraz1} yields the desired estimate.
\end{proof}

\begin{proposition}\label{prop:sols_linearization_eq}
Consider a Riemannian metric in $\overline{\R_+^n}$ of the form $g(x) =\exp(h(x))$, where $h(x)$ is a trace-free symmetric two-tensor in $\overline{\R_+^n}$ satisfying $|h(x)| + |\d h(x)| + |{\d}^2h(x)|\leq\alpha\leq 1$ and $h_{na}(x)=0$ for all $x\in \overline{\Rn}$ and $h(x)=0$ for all $x \in\overline{\Rn}\setminus B_1^+(0)$. Here $\alpha$ depends only on $n$ and $T_c$. Then, given any pair $\ks\in\R^{n-1}\times(0,\infty)$
 and any functions $f\in L^{\frac{2n}{n+2}}(\Rn),\hat f\in L^{\frac{2(n-1)}{n}}(\d\Rn)$, there exists a unique function $w := G\ks(f,\hat f)\in \mathcal{E} _{\ks}$, such that
\begin{align}\label{linearization}
&\int_{\Rn}\left(\langle \nabla w,\nabla \varphi\rangle_g+c_nR_g w\varphi-n(n+2)u_{\ks}^{\frac{4}{n-2}}w\varphi\right)+\int_{\d \Rn}\left(d_nh_g\varphi+nT_cu_{\ks}^{\frac{2}{n-2}}\varphi\right)w\no\\
=&\int_{\Rn}f\varphi+\int_{\d\Rn}\hat f\varphi
\end{align}
 for all $\varphi \in \mathcal{E}_{\ks}$. Furthermore, there holds
 \begin{align}\label{linear esti}
 \|G\ks(f,\hat f)\|_{\mathcal{E}}\leq C\|f\|_{L^{\frac{2n}{n+2}}(\Rn)}+C\|\hat f\|_{L^{\frac{2(n-1)}{n}}(\d\Rn)}.
\end{align}
 \end{proposition}
 \begin{proof}
 By Propositions \ref{prop:eigenvalues} and \ref{prop:conf_operators}, and H\"older's inequality, we can follow nearly the same lines in \cite[Corollary 3]{Brendle3} that there exist two positive constants $\alpha$ and $C$, depending only on $n$ and $T_c$, such that $|h(x)| + |\d h(x)| + |{\d}^2h(x)|\leq\alpha$ and 
\begin{align}\label{est:coercive}
\frac{\mu}{2} \|w\|_{\mathcal{E}}^2\leq&\int_{\Rn}\left(|\nabla w|_g^2 +c_n R_g w^2-n(n+2)u_{\ks}^{\frac{4}{n-2}}w^2\right)\no\\
&+\int_{\d\Rn}\left(d_n h_g w^2+n T_c u_{\ks}^{\frac{2}{n-2}}w^2\right)+\frac{C}{\mu}\mathcal{A}(w)^2
\end{align}
for all $w \in \mathcal{E}_{\ks}$, where
\begin{align*}
\mathcal{A}(w)=&\int_{\Rn}\left(\Delta_g u_{\ks}-c_n R_g u_{\ks}+n(n+2) u_{\ks}^{\frac{n+2}{n-2}}\right)w\\
&- \int_{\d\Rn}\left(d_n h_g u_{\ks}+2 T_cu_{\ks}^{\frac{n}{n-2}}\right)w.
\end{align*}

Suppose that $w\in\mathcal{E}_{\ks}$ satisfies \eqref{linearization}, then
\begin{align}\label{eq:weak_sol_w}
&\int_{\Rn}\left(|\nabla w|_g^2+c_nR_g w^2-n(n+2)u_{\ks}^{\frac{4}{n-2}}w^2\right)+\int_{\d \Rn}\left(d_nh_gw^2+nT_cu_{\ks}^{\frac{2}{n-2}}w^2\right)\no\\
=&\int_{\Rn}fw+\int_{\d\Rn}\hat fw.
\end{align}
Since $u_{\ks}\in \mathcal{E}_{\ks}$, we have
\begin{align*}
&\int_{\Rn}\left(\Delta_g u_{\ks}-c_nR_g u_{\ks}+n(n+2)u_{\ks}^{\frac{n+2}{n-2}}\right)w\\
&-\int_{\d \Rn}\left(d_nh_g u_{\ks}+2T_cu_{\ks}^{\frac{n}{n-2}}\right)w=-\int_{\Rn}fu_{\ks}-\int_{\d\Rn}\hat fu_{\ks}.
\end{align*}
Then by \eqref{ineq: Sobolev and trace}, \eqref{est:coercive} and \eqref{eq:weak_sol_w}, we have
\begin{align*}
&\frac{\mu}{2}\|w\|_{\mathcal{E}}^{2}\\
\leq& \int_{\Rn}fw+\int_{\d\Rn}\hat fw+\frac{C}{\mu}\mathcal{A}(w)^2\\
\leq&\|f\|_{L^{\frac{2n}{n+2}}(\Rn)}\|w\|_{L^{\crit}(\Rn)}+\|\hat f\|_{L^{\frac{2(n-1)}{n}}(\d\Rn)}\|w\|_{L^{\frac{2(n-1)}{n-2}}(\d\Rn)}\\
&+\frac{C}{\mu}\left(\|f\|_{L^{\frac{2n}{n+2}}(\Rn)}^{2}\|u_{\ks}\|_{L^{\crit}(\Rn)}^{2}+\|\hat f\|_{L^{\frac{2(n-1)}{n}}(\d\Rn)}^{2}\|u_{\ks}\|_{L^{\frac{2(n-1)}{n-2}}(\d\Rn)}^{2} \right)\\
\leq& K^{\frac{1}{2}}\|w\|_{\mathcal{E}}\left(\|f\|_{L^{\frac{2n}{n+2}}(\Rn)}+\|\hat f\|_{L^{\frac{2(n-1)}{n}}(\d\Rn)}\right)+\frac{C}{\mu}\left(\|f\|_{L^{\frac{2n}{n+2}}(\Rn)}^{2}+\|\hat f\|_{L^{\frac{2(n-1)}{n}}(\d\Rn)}^{2}\right).
\end{align*}
Hence it follows from Young's inequality that
 \begin{align*}
 &\frac{\mu}{4}\|w\|_{\mathcal{E}}^{2}\\
 \leq& \frac{2K}{\mu}\left(\|f\|_{L^{\frac{2n}{n+2}}(\Rn)}^{2} +\|\hat f\|_{L^{\frac{2(n-1)}{n}}(\d\Rn)}^{2}\right)+\frac{C}{\mu}\left(\|f\|_{L^{\frac{2n}{n+2}}(\Rn)}^{2} +\|\hat f\|_{L^{\frac{2(n-1)}{n}}(\d\Rn)}^{2}\right).
 \end{align*}
This implies the uniqueness of the solutions to \eqref{linearization}.

For the existence part, thanks to the coercive estimate \eqref{est:coercive}, it suffices to minimize the following functional
\begin{align*}
&\int_{\Rn}\left(|\nabla w|_g^2+c_n R_g w^2-n(n+2)u_{\ks}^{\frac{4}{n-2}}w^2-2fw\right)\\
&+\int_{\d\Rn}\left(d_n h_g w^2+n T_c u_{\ks}^{\frac{2}{n-2}}w^2
-2\hat fw\right)+\frac{C}{\mu}\mathcal{A}(w)^2
\end{align*}
 over all $w \in \mathcal{E}_{(\xi,\e)}$.
 \end{proof}
 
 \begin{proposition}\label{prop:sol_v}
 Consider a Riemannian metric in $\overline{\R_+^n}$ of the form $g(x) =\exp(h(x))$, where $h(x)$ is a trace-free symmetric two-tensor in $\overline{\R_+^n}$ satisfying $|h(x)| + |\d h(x)| + |{\d}^2h(x)|\leq\alpha\leq 1$ and $h_{na}(x)=0$ for all $x\in \overline{\Rn}$ and $h(x)=0$ for all $x \in\overline{\Rn}\setminus B_1^+(0)$. Here $\alpha$ depends only on $n$ and $T_c$. Then, given any pair $\ks\in\R^{n-1}\times(0,\infty)$, there exists a unique function $v_{\ks}\in \mathcal{E}$ such that $v_{\ks}-u_{\ks}\in \mathcal{E}_{\ks}$,\,and
\begin{align}\label{annihilator}
 &\int_{\Rn}\left(\langle \nabla v_{\ks},\nabla\varphi\rangle_g+c_nR_g v_{\ks}\varphi-n(n-2)|v_{\ks}|^{\frac{4}{n-2}}v_{\ks}\varphi\right)\no\\
 &+\int_{\d \Rn}\left(d_nh_gv_{\ks}+(n-2)T_c|v_{\ks}|^{\frac{2}{n-2}}v_{\ks}\right)\varphi =0
 \end{align}
 for all $\varphi\in \mathcal{E}_{\ks}$. Moreover, there exists a positive constant $C$, depending only on $T_c$ and $n$, such that
 \begin{align}\label{annihilator esti}
 &\|v_{\ks}-u_{\ks}\|_{\mathcal{E}}\no\\
 \leq& C \|\Delta_{g}u_{\ks}-c_nR_gu_{\ks}+n(n-2)u_{\ks}^{\frac{n+2}{n-2}}\|_{L^{\conj}(\Rn)}+C\|h_g u_{\ks}\|_{L^{\frac{2(n-1)}{n}}(\d\Rn)}.
 \end{align}
 In particular, $v_{\ks}\not \equiv 0$ if $\alpha$ is sufficiently small.
 \end{proposition}
 
 \begin{proof}
 Let $G_{\ks}:L^{\frac{2n}{n+2}}(\Rn)\times L^{\frac{2(n-1)}{n}}(\d\Rn)\to \mathcal{E}_{\ks}$ be the solution operator constructed in Proposition \ref{prop:sols_linearization_eq}, and we define a nonlinear operator $\Phi_{\ks}$ on $\mathcal{E}_{\ks}$ by
 \begin{align*}
 &\Phi_{\ks}(w)\\
 =&G_{\ks}\left(\Delta_g u_{\ks}-\frac{n-2}{4(n-1)}R_g u_{\ks}+n(n-2)u_{\ks}^{\frac{n+2}{n-2}},-\frac{n-2}{2}h_g u_{\ks} \right)+\\
 &G_{\ks}\left( n(n-2)|u_{\ks}+w|^{\frac{4}{n-2}}(u_{\ks}+w)-n(n-2)u_{\ks}^{\frac{n+2}{n-2}}-n(n+2)u_{\ks}^{\frac{4}{n-2}}w,\right.\\
 & \qquad~~\left.-(n-2)T_c|u_{\ks}+w|^{\frac{2}{n-2}}(u_{\ks}+w)+(n-2)T_cu_{\ks}^{\frac{n}{n-2}}+n T_c u_{\ks}^{\frac{2}{n-2}}w \right).
 \end{align*}
 In particular, it follows from Propositions \ref{prop:conf_operators} and \ref{prop:sols_linearization_eq}  that $\|\Phi_{\ks}(0)\|_{\mathcal{E}}\leq C\alpha $.
 
 Using the pointwise estimates
 \begin{align*}
 &\left||u_{\ks}+w|^{\frac{4}{n-2}}(u_{\ks}+w)-|u_{\ks}+\widetilde{w}|^{\frac{4}{n-2}}(u_{\ks}+\widetilde{w})-\frac{n+2}{n-2}u_{\ks}^{\frac{4}{n-2}}(w-\widetilde{w})\right|\\
 \leq& C(|w|^{\frac{4}{n-2}}+|\widetilde{w}|^{\frac{4}{n-2}})|w-\widetilde{w}|
 \end{align*}
 and
 \begin{align*}
  &\left||u_{\ks}+w|^{\frac{2}{n-2}}(u_{\ks}+w)-|u_{\ks}+\widetilde{w}|^{\frac{2}{n-2}}(u_{\ks}+\widetilde{w})-\frac{n}{n-2} u_{\ks}^{\frac{2}{n-2}}(w-\widetilde{w})\right|\\
 \leq& C(|w|^{\frac{2}{n-2}}+|\widetilde{w}|^{\frac{2}{n-2}})|w-\widetilde{w}|
 \end{align*}
 and Proposition \ref{prop:sols_linearization_eq}, we obtain
 \begin{align*}
 &\|\Phi_{\ks}(w)-\Phi_{\ks}(\widetilde{w})\|_{\mathcal{E}}\\
 \leq& C\left\||u_{\ks}+w|^{\frac{4}{n-2}}(u_{\ks}+w)-|u_{\ks}+\widetilde{w}|^{\frac{4}{n-2}}(u_{\ks}+\widetilde{w})\right.\\
 &\qquad\left.-\frac{n+2}{n-2}u_{\ks}^{\frac{4}{n-2}}(w-\widetilde{w})\right\|_{L^{\frac{2n}{n+2}}(\Rn)}\\
 &+C\left\||u_{\ks}+w|^{\frac{2}{n-2}}(u_{\ks}+w)-|u_{\ks}+\widetilde{w}|^{\frac{2}{n-2}}(u_{\ks}+\widetilde{w})\right.\\
 &\qquad\quad\left.-\frac{n}{n-2}u_{\ks}^{\frac{2}{n-2}}(w-\widetilde{w})\right\|_{L^{\frac{2(n-1)}{n}}(\d\Rn)}\\
 \leq& C\left(\|w\|_{L^{\crit}(\Rn)}^{\frac{4}{n-2}}+\|\widetilde{w}\|_{L^{\crit}(\Rn)}^{\frac{4}{n-2}}\right)\|w-\widetilde{w}\|_{\mathcal{E}}\\
 &+ C\left(\|w\|_{L^{\frac{2(n-1)}{n-2}}(\d\Rn)}^{\frac{2}{n-2}}+\|\widetilde{w}\|_{L^{\frac{2(n-1)}{n-2}}(\d\Rn)}^{\frac{2}{n-2}}\right)\|w-\widetilde{w}\|_{\mathcal{E}}
 \end{align*}
 for $w,\widetilde{w}\in\mathcal{E}_{\ks} $.
 Hence, if $\alpha$ is sufficiently small, then the contraction mapping principle implies that $\Phi_{\ks}$ has a unique fixed point $w_{\ks}$ within $\mathcal{E}_{\ks}$. Hence $v_{\ks}=u_{\ks}+w_{\ks}$ is the desired solution, and not identically zero, which follows from \eqref{annihilator esti} and Proposition \ref{prop:conf_operators}. 
 \end{proof}
 
 Given a pair $\ks \in \R^{n-1}\times(0,\infty)$, we define the following energy functional 
 \begin{align}\label{energy functional}
& \mathcal{F}_{g}\ks\no\\
:=&\int_{\Rn}\left(|\nabla v_{\ks}|_{g}^{2}+c_nR_gv_{\ks}^{2}-(n-2)^2|v_{\ks}|^{\crit}\right)+d_n\int_{\d\Rn} h_g v_{\ks}^{2}\no\\
&+\frac{(n-2)^2}{n-1}T_c\int_{\d\Rn}|v_{\ks}|^{\frac{2(n-1)}{n-2}}-2(n-2)\int_{\Rn}u_{\ks}^{\crit}+\frac{n-2}{n-1}T_c\int_{\d\Rn}u_{\ks}^{\frac{2(n-1)}{n-2}}.
 \end{align}

 \begin{proposition}\label{prop:v_critical_pt_F}
 The function $\mathcal{F}_{g}$ is continuously differentiable. Moreover, if $\ks$ is a critical point of $\mathcal{F}_{g}$, then the function $v_{\ks} $\,is a positive smooth solution of 
\begin{align}
\begin{cases}
\displaystyle -\D_{g}v_{\ks}+c_n R_g v_{\ks}=n(n-2)v_{\ks}^{\frac{n+2}{n-2}}, &\quad \mathrm{in~~} \Rn, \\
\displaystyle \frac{\d v_{\ks}}{\d x_n}-d_n h_g v_{\ks}=(n-2)T_c v_{\ks}^{\frac{n}{n-2}}, &\quad \mathrm{on~~} \d \Rn.
\end{cases}
\end{align}
 \end{proposition}
 
\begin{proof}
By definition of $v_{\ks}$, we can find real numbers $b_a=b_a\ks, 1 \leq a \leq n$, such that
\begin{align*}
 &\int_{\Rn}\left(\langle \nabla v_{\ks},\nabla \varphi\rangle_g+c_nR_g v_{\ks}\varphi-n(n-2)|v_{\ks}|^{\frac{4}{n-2}}v_{\ks}\varphi\right)\no\\
 &+\int_{\d \Rn}\left(d_nh_gv_{\ks}\varphi+(n-2)T_c|v_{\ks}|^{\frac{2}{n-2}}v_{\ks}\varphi\right) \\
 =&\sum_{a=1}^{n}b_a\ks \left(2n\int_{\Rn}\varphi u_{\ksa}-T_c\int_{\d\Rn}\varphi\hat{u}_{\ksa}\right)
\end{align*}
for all test function\,$\varphi\in\mathcal{E}$. This implies
\begin{align*}
\d_\epsilon\mathcal{F}_g=2\sum_{a=1}^{n} b_a\ks \left(2n\int_{\Rn}\d_\epsilon v_{\ks} u_{\ksa}-T_c\int_{\d\Rn}\d_\epsilon v_{\ks}\hat{u}_{\ksa}\right)
\end{align*}
and
\begin{align*}
\d_{\xi_j}\mathcal{F}_g=2\displaystyle{\sum_{a=1}^{n}}b_a\ks \left(2n\int_{\Rn}\d_{\xi_j} v_\ks u_{\ksa}-T_c\int_{\d\Rn}\d_{\xi_j}v_\ks\hat{u}_{\ksa}\right)
\end{align*}
for $1 \leq j \leq n-1$. On the other hand, we have
\begin{align*}
 0=2n\int_{\Rn}(v_{\ks}-u_{\ks}) u_{\ksa}-T_c\int_{\d\Rn}(v_{\ks}-u_{\ks})\hat{u}_{\ksa},
\end{align*}
since $v_{\ks}-u_{\ks}\in\mathcal{E}_{\ks}$. Differentiating the above equation with respect to $\epsilon$ and $\xi_k$, we obtain
\begin{align*}
0=&2n\int_{\Rn}\d_\epsilon(v_{\ks}-u_{\ks}) u_{\ksa}-T_c\int_{\d\Rn}\d_\epsilon(v_{\ks}-u_{\ks})\hat{u}_{\ksa}\\
&+2n\int_{\Rn}(v_{\ks}-u_{\ks}) \d_\epsilon u_{\ksa}-T_c\int_{\d\Rn}(v_{\ks}-u_{\ks})\d_\epsilon\hat{u}_{\ksa}\\
=&2n\int_{\Rn}(v_{\ks}-u_{\ks}) \d_\epsilon u_{\ksa}-T_c\int_{\d\Rn}(v_{\ks}-u_{\ks})\d_\epsilon\hat{u}_{\ksa}\\
&+2n\int_{\Rn}u_{\ksa}\d_\epsilon v_{\ks} -T_c\int_{\d\Rn}\hat{u}_{\ksa}\d_\epsilon v_{\ks}+\bar c_n(2\epsilon)^{-1}\delta_{na},
\end{align*}
where  $1\leq a\leq n$, $\bar c_n$ is a nonzero constant independent of $\xi$ and $\epsilon$, and
\begin{align*}
0=&2n\int_{\Rn}(v_{\ks}-u_{\ks}) \d_{\xi_k}u_{\ksa}-T_c\int_{\d\Rn}(v_{\ks}-u_{\ks})\d_{\xi_k}\hat{u}_{\ksa}\\
&+2n\int_{\Rn}\d_{\xi_k}v_{\ks} u_{\ksa}-T_c\int_{\d\Rn}\d_{\xi_k}v_{\ks}\hat{u}_{\ksa}-\bar c_a(2\epsilon)^{-1}\delta_{ak},
\end{align*}
where each $\bar c_i, 1 \leq i \leq n-1$ is a nonzero constant independent of $\xi$ and $\epsilon$.

Therefore, putting these facts together, we conclude that
\begin{align*}
&-\bar c_n b_n\ks=\epsilon\d_\epsilon\mathcal{F}_g+\\
&2\epsilon \sum_{a=1}^{n}b_a\ks \left[2n\int_{\Rn}\d_\epsilon u_{\ksa} (v_{\ks}-u_{\ks})-T_c\int_{\d\Rn}\d_\epsilon \hat{u}_{\ksa}(v_{\ks}-u_{\ks})\right]
\end{align*}
and for $1 \leq i \leq n-1$,
\begin{align*}
&\bar c_i a_i\ks=\epsilon\d_{\xi_i}\mathcal{F}_g+\\
&2\epsilon\sum_{a=1}^{n}b_a\ks \left[2n\int_{\Rn}\d_{\xi_i} u_{\ksa} (v_{\ks}-u_{\ks})-T_c\int_{\d\Rn}\d_{\xi_i}\hat{u}_{\ksa}(v_{\ks}-u_{\ks})\right].
\end{align*}
Hence, if $\ks$ is a critical point of $\mathcal{F}_g$, then
\begin{align*}
&\sum_{a=1}^{n}|\bar c_a||b_a\ks|\\
\leq& C\left( \|v_{\ks}-u_{\ks}\|_{L^{\crit}(\Rn)}+\|v_{\ks}-u_{\ks}\|_{L^{\frac{2(n-1)}{n-2}}(\d\Rn)}\right){\sum_{a=1}^{n}}|b_a\ks|.
\end{align*}
By \eqref{ineq: Sobolev and trace} and \eqref{annihilator esti} we have
\begin{align*}
 &\|v_{\ks}-u_{\ks}\|_{L^{\crit}(\Rn)}+\|v_{\ks}-u_{\ks}\|_{L^{\frac{2(n-1)}{n-2}}(\d\Rn)}\\
 \leq& C  \|v_{\ks}-u_{\ks}\|_{\mathcal{E}}\leq C\alpha.
\end{align*}
Thus, if $\alpha$ is sufficiently small, we obtain 
\begin{align*}
\displaystyle{\sum_{a=1}^{n}}|b_a\ks|=0.
\end{align*}
Consequently, we have 
\begin{align*}
&\int_{\Rn}\left(\langle \nabla v_{\ks}, \nabla \varphi\rangle_g+c_nR_g v_{\ks}\varphi-n(n-2)|v_{\ks}|^{\frac{4}{n-2}}v_{\ks}\varphi\right)\no\\
 &+\int_{\d \Rn}\left(d_nh_gv_{\ks}\varphi+(n-2)T_c|v_{\ks}|^{\frac{2}{n-2}}v_{\ks}\varphi\right)=0
\end{align*}
for all $\varphi\in \mathcal{E}$.

Finally, we follow the same lines in \cite[Proposition 6]{Brendle3} that $v_{\ks}\geq 0$ in $\overline{\Rn}$. Together with $v_{\ks}\not \equiv 0$ by Proposition \ref{prop:sol_v}, the strong maximum principle and the Hopf boundary point lemma give $v_{\ks}>0$ in $\overline{\Rn}$. By the regularity theory of P. Cherrier \cite{cherrier}, we show that $v_{\ks}$ is smooth.
\end{proof}

\section{An estimate for the energy of a bubble}\label{Sect4}

We first introduce a multi-linear form $\overline W : \R^{n-1} \times \mathbb{R}^{n-1} \times \mathbb{R}^{n-1} \times \mathbb{R}^{n-1} \to \R$ satisfying the same algebraic properties of the Weyl tensor on $\pa \Rn$. Moreover, we assume
$$\sum\limits_{i,j,k,l=1}^{n-1}(\overline{W}_{ikjl}+\overline{W}_{iljk})^2>0.$$
If $x=(x',x_n)\in\Rn$, then we identify $x'$ with $(x',0)\in\pa \Rn$ and  define
\begin{equation}\label{def:H_ab}
H_{ij}(x)=H_{ij}(x')=\overline{W}_{ikjl}x^k x^l \quad\mathrm{~~and~~}\quad H_{na}(x)=0,
\end{equation}
as well as $\overline H_{ab}(x)=f(|x'|^2)H_{ab}(x)$,
where $f(s)$ is a polynomial of degree $d$ for $0\leq d<(n-6)/4$ and is to be determined later. Then $H$ is symmetric, trace-free, independent of the variable $x_n$, and satisfies
\begin{align*}
x^a H_{ab}(x)=x^i H_{ib}(x)=\d_{a}H_{ab}(x)=\d_{i}H_{ib}(x)=0.
\end{align*}
We define a Riemannian metric $g = \exp(h)$ in $\overline{\Rn}$, where $h$ is a trace-free symmetric two-tensor in $\overline{\Rn}$ and $h_{na}=\d_{n}h_{ab}(x)=0$ for all $x \in \overline{\Rn}$, and satisfies
\begin{align*}
\begin{cases}
h_{ab}(x)=\mu\lambda^{2d}f(\lambda^{-2}|x'|^2)H_{ab}(x),&\quad\mathrm{if~~} |x|\leq \rho,\\
h_{ab}(x)=0,&\quad\mathrm{if~~} |x|\geq 1.
\end{cases}
\end{align*}
Here $0<\mu \leq 1, 0<\lambda\leq\rho\leq 1$. This gives $h_{ab}(x)=O(\mu(\lambda+|x|)^{2d+2})$. In addition, we require that $|h|+|\d h|+|\d^{2}h|\leq \alpha$, where $\alpha$ is the constant given in Proposition \ref{prop:sol_v}. The boundary $\pa \Rn$ is totally geodesic with respect to $g$, since the second fundamental form vanishes on $\d\Rn$ , explicitly
\begin{align*}
\pi_{ij}=\Gamma_{ij}^{n}=\frac{1}{2}\left(\frac{\pa g_{in}}{\pa x_j}+\frac{\pa g_{jn}}{\pa x_i}-\frac{\pa g_{ij}}{\pa x_n}\right)=0.
\end{align*}
Applying Proposition \ref{prop:sol_v} to each pair $\ks\in \R^{n-1}\times(0,\infty)$, we choose $v_{\ks}$ to be the unique element of $\mathcal{E}$ such that $v_{\ks}-u_{\ks}\in\mathcal{E}_{\ks}$ and 
\begin{align*}
 &\int_{\Rn}\left(\langle \nabla v_{\ks},\nabla \varphi\rangle_g+c_nR_g v_{\ks}\varphi-n(n-2)|v_{\ks}|^{\frac{4}{n-2}}v_{\ks}\varphi\right)\no\\
 &+(n-2)T_c\int_{\d \Rn}|v_{\ks}|^{\frac{2}{n-2}}v_{\ks}\varphi =0
 \end{align*}
 for all $\varphi\in \mathcal{E}_{\ks}$.
 
Let $\Omega=\left\{\ks\in \R^{n-1}\times(0,\infty);|\xi|<1,\frac{1}{2}<\epsilon<1 \right\}$. Similar to \cite[Proposition 7 and Corollary 8]{Brendle3} and \cite[Proposition 5 and Corollary 6]{Brendle-Marques},  for any pair $\ks\in\lambda\Omega$ we obtain
 \begin{align}\label{asym esti1}
 \|\Delta_{g}u_{\ks}-c_n R_g u_{\ks}+n(n-2)u_{\ks}^{\frac{n+2}{n-2}}\|_{L^{\frac{2n}{n+2}}(\Rn)}\leq C\mu \lambda^{2d+2}+C\lambda^{\frac{n-2}{2}}\rho^{\frac{2-n}{2}},
 \end{align}
 \begin{align}\label{asym esti2}
 &\|\Delta_{g}u_{\ks}-c_n R_g u_{\ks}+n(n-2)u_{\ks}^{\frac{n+2}{n-2}}+\mu \lambda^{2d}f(\lambda^{-2}|x'|^2)H_{ij}(x)\d_i\d_j u_{\ks}\|_{L^{\frac{2n}{n+2}}(\Rn)}\no\\
 \leq& C\mu^{2} \lambda^{\frac{(2d+2)(n+2)}{n-2}}+C\lambda^{\frac{n-2}{2}}\rho^{\frac{2-n}{2}}
 \end{align}
 and together with Proposition \ref{prop:sol_v},
 \begin{align}\label{sol esti}
& \|v_{\ks}-u_{\ks}\|_{L^{\crit}(\Rn)}+\|v_{\ks}-u_{\ks}\|_{L^{\frac{2(n-1)}{n-2}}(\d\Rn)}\no\\
 \leq& C\mu \lambda^{2d+2}+C\lambda^{\frac{n-2}{2}}\rho^{\frac{2-n}{2}}.
 \end{align}
 
 By Proposition \ref{prop:sols_linearization_eq} with $h=0$, we define the function $w_{\ks}$ as the unique element of $\mathcal{E}_{\ks}$ satisfying
 \begin{align}\label{linear approximation}
 &\int_{\Rn}\left(\langle \nabla w_{\ks},\nabla \varphi\rangle-n(n+2)u_{\ks}^{\frac{4}{n-2}}w_{\ks}\varphi\right)+n T_c\int_{\d\Rn} u_{\ks}^{\frac{2}{n-2}}w_{\ks}\varphi\no\\
 =&-\int_{\Rn}\mu\lambda^{2d}f(\lambda^{-2}|x'|^2)H_{ij}(x)\d_i\d_ju_{\ks}\varphi
 \end{align}
 for all $\varphi\in \mathcal{E}_{\ks}$. In particular, $w_{(0,\epsilon)}=0$, since $x^iH_{ij}(x)=0$ for any $x\in \overline{\Rn}$.
 \begin{proposition}\label{prop:est_w}
 The function $w_{\ks}$ is smooth and satisfies that given any $\ks \in \lambda\Omega$,
 $|\d^{k}w_{\ks}(x)|\leq C\lambda^{\frac{n-2}{2}}\mu (\lambda+|x|)^{2d+4-k-n}$ for all $x\in \Rn, k=0,1,2$.
  \end{proposition}
 \begin{proof}
 By definition of $\mathcal{E}_{\ks}$, there exist real numbers $\bar b_a\ks, 1\leq a\leq n$ such that
 \begin{align}\label{linearize expansion}
 &\int_{\Rn}\left(\langle\nabla w_{\ks},\nabla \phi \rangle -n(n+2)u_{\ks}^{\frac{4}{n-2}}w_{\ks} \phi\right)+nT_c\int_{\d\Rn}u_{\ks}^{\frac{2}{n-2}}w_{\ks}\phi\no \\
 =&-\int_{\Rn}\mu \lambda^{2d}f(\lambda^{-2}|x'|^2)H_{ij}(x)\d_i\d_j u_{\ks}(x)\phi\no\\
 &+\sum_{a=1}^{n}\bar b_a\ks\left(2n\int_{\Rn}u_{\ksa}\phi-T_c\int_{\d\Rn}\hat{u}_{\ksa}\phi \right )
 \end{align}
 for all $\phi \in \mathcal{E}$. Hence it follows from standard elliptic estimates that $w_{\ks}$ is smooth.
 Since 
 \begin{align*}
 \| \mu \lambda^{2d}f(\lambda^{-2}|x'|^2)H_{ij}(x)\d_i\d_j u_{\ks}(x)\|_{L^{\frac{2n}{n+2}}(\Rn)}\leq C\mu\lambda^{2d+2},
 \end{align*}
 then by \eqref{ineq: Sobolev and trace} and \eqref{linear esti} we have 
 \begin{align*}
 \|w_{\ks}\|_{L^{\frac{2n}{n-2}}(\Rn)}+\|w_{\ks}\|_{L^{\frac{2(n-1)}{n-2}}(\d\Rn)}\leq C\|w_{\ks}\|_{\mathcal{E}}\leq C\mu \lambda^{2d+2}.
 \end{align*}
 Choosing $\phi=u_{\ksa}$ in $\eqref{linearize expansion}$, we obtain 
 \begin{align*}
 \displaystyle{\sum_{a=1}^{n}}|\bar b_a \ks|\leq C\mu\lambda^{2d+2}.
 \end{align*}
 Hence, we have
 \begin{align*}
&\left |\Delta w_{\ks}+n(n+2)u_{\ks}^{\frac{4}{n-2}}w_{\ks}\right |\\
=&\left|\mu \lambda^{2d}f(\lambda^{-2}|x'|^2)H_{ij}(x)\d_i\d_j u_{\ks}(x)-2n\bar b_a\ks u_{\ksa}\right|\\
 \leq&C \mu\lambda^{\frac{n-2}{2}}(\lambda+|x|)^{2d+2-n}
 \end{align*}
 for all $x\in \Rn$, and 
 \begin{align*}
 \left|\frac{\d}{\d x_n}w_{\ks}+nT_c u_{\ks}^{\frac{2}{n-2}}w_{\ks}\right|=\left|-T_c\sum_{a=1}^{n}\bar b_a \ks \hat u_{\ksa}\right|\leq C\mu\lambda^{\frac{n}{2}}(\lambda+|x|)^{2d+2-n}
 \end{align*}
 for all $x\in \d\Rn$.
 We let $r=(\lambda+|x_0|)/2$ for any fixed $x_0\in \overline{\Rn}$. Then $\lambda +|x| \geq r$ for any $x \in B_r^+(x_0)$. Based on the above facts, we obtain 
 \begin{align*}
 u_{\ks}^{\frac{2}{n-2}}(x)\leq& C r^{-1},  &\quad \forall~ x\in B_{r}^{+}(x_0),\\
 \left|\frac{\d}{\d x_n}w_{\ks}-nT_cu_{\ks}^{\frac{2}{n-2}}w_{\ks}\right|\leq& C\mu\lambda^{\frac{n}{2}}r^{2d+2-n}, &\quad \forall~ x\in B_{r}^{+}(x_0)\cap \d\Rn,\\
 \left |\Delta w_{\ks}+n(n+2)u_{\ks}^{\frac{4}{n-2}}w_{\ks}\right|\leq& C\mu\lambda^{\frac{n-2}{2}}r^{2d+2-n}, &\quad \forall~ x\in B_{r}^{+}(x_0).
 \end{align*}
 By \cite[Theorems 8.25 and 8.26] we have 
 \begin{align*}
 &r^{\frac{n-2}{2}}|w_{\ks}(x_0)|\\
 \leq& C\|w_{\ks} \|_{L^{\frac{2n}{n-2}}(B_{r}^{+}(x_0))}+Cr^{\frac{n+2}{2}}\left\|\Delta w_{\ks}+n(n+2)u_{\ks}^{\frac{4}{n-2}}w_{\ks}\right\|_{L^{\infty}(B_{r}^{+}(x_0))}\\
 &+Cr^{\frac{n}{2}}\left\|\frac{\d}{\d x_n}w_{\ks}+nT_cu_{\ks}^{\frac{2}{n-2}}w_{\ks}\right\|_{L^{\infty}(B_{r}^{+}(x_0)\cap\d\Rn)}\\
 \leq& C\mu\lambda^{2d+2}+C\mu\lambda^{\frac{n-2}{2}}r^{2d+2+\frac{2-n}{2}}+C\mu\lambda^{\frac{n}{2}}r^{2d+2-\frac{n}{2}}\\
 \leq& C\mu\lambda^{2d+2},
 \end{align*}
 since $d<(n-6)/4$. Then we obtain
 \begin{align*}
 \underset{x\in \Rn}{\sup}(\lambda+|x|)^{\frac{n-2}{2}}|w_{\ks}(x)|\leq C\mu\lambda^{2d+2}.
 \end{align*}
 By Green's representation formula, we have 
 \begin{align*}
 &w_{\ks}(x)\\
 =&-\frac{1}{(n-2)|\Sp^{n-1}|}\int_{\Rn}(|x-y|^{2-n}+|x^\ast-y|^{2-n})\Delta w_{\ks}(y)\ud y\\
 &-\frac{1}{(n-2)|\Sp^{n-1}|}\int_{\d\Rn}(|x-y|^{2-n}+|x^\ast-y|^{2-n})\frac{\d}{\d y_n}w_{\ks}(y)\ud y ,
 \end{align*}
 for any $x \in \Rn$, where $x^\ast=(x_1,x_2,\cdots,x_{n-1},-x_n)$. From this we obtain 
 \begin{align*}
 &\underset{x\in\Rn}{\sup}(\lambda+|x|)^{\beta}|w_{\ks}(x)|\\
 &\leq C\underset{x\in\Rn}{\sup}(\lambda+|x|)^{\beta+2}|\Delta w_{\ks}(x)|+C\underset{x\in\d\Rn}{\sup}(\lambda+|x|)^{\beta+1}\left|\frac{\d}{\d x_n}w_{\ks}(x)\right|
 \end{align*}
 for all $0<\beta<n-2$. Since 
 \begin{align*}
 |\Delta w_{\ks}(x)|\leq n(n+2) u_{\ks}^{\frac{4}{n-2}}(x)|w_{\ks}(x)|+C\mu\lambda^{\frac{n-2}{2}}(\lambda+|x|)^{2d+2-n}  ,& \quad \forall ~x\in \Rn,
 \end{align*}
 and
 \begin{align*}
 \left|\frac{\d}{\d x_n}w_{\ks}(x)\right|\leq n |T_c| u_{\ks}^{\frac{2}{n-2}}(x)|w_{\ks}(x)|+C\mu\lambda^{\frac{n}{2}}(\lambda+|x|)^{2d+2-n} , &\quad \forall ~x\in \d\Rn,
 \end{align*}
 we conclude that 
 \begin{align*}
 &\underset{x\in\Rn}{\sup}(\lambda+|x|)^{\beta}|w_{\ks}(x)|\\
 \leq& C\lambda^{2}\underset{x\in\Rn}{\sup}(\lambda+|x|)^{\beta-2}|w_{\ks}(x)|+C\lambda \underset{x\in\d\Rn}{\sup}(\lambda+|x|)^{\beta-1}|w_{\ks}(x)|+C\mu\lambda^{\beta+2d+3-\frac{n}{2}}
 \end{align*}
 for all $0<\beta\leq n-4-2d$. Iterating this inequality, we obtain
 \begin{align*}
 \underset{x\in\Rn}{\sup}(\lambda+|x|)^{n-2d-4}|w_{\ks}(x)|\leq C\mu\lambda^{\frac{n-2}{2}} .
 \end{align*}
 Differentiating the equation \eqref{linear approximation} twice and repeating the argument above, we obtain the estimates of the  first and second derivatives of $w_{\ks}$.  
 \end{proof}

\begin{proposition}\label{prop:est_v-u-w}
There holds
  \begin{align*}
 &\|v_{\ks}-u_{\ks}-w_{\ks}\|_{L^{\crit}(\Rn)}+\|v_{\ks}-u_{\ks}-w_{\ks}\|_{L^{\frac{2(n-1)}{n-2}}(\d\Rn)}\no\\
 \leq& C\mu^{\frac{n}{n-2}}\lambda^{\frac{(2d+2)n}{n-2}}+C\lambda^{\frac{n-2}{2}}\rho^{\frac{2-n}{2}}
 \end{align*}
 for all $\ks\in \lambda\Omega$.
 \end{proposition}
 
 \begin{proof}
 Consider the functions
 \begin{align*}
 B_1=\displaystyle{\sum_{i,k=1}^{n}}\d_i[(g^{ik}-\delta_{ik})\d_k w_{\ks}]-c_n R_g w_{\ks}
 \end{align*}
and 
\begin{align*}
B_2=\displaystyle{\sum_{i,k=1}^{n}}\mu\lambda^{2d} f(\lambda^{-2}|x|^2)H_{ik}(x)\d_i\d_k u_{\ks}.
\end{align*}
By definition of $w_{\ks}$, we have 
\begin{align*}
&\int_{\Rn}\left(\langle \nabla w_{\ks},\nabla\varphi\rangle_{g}+c_n R_g w_{\ks}\phi-n(n+2)u_{\ks}^{\frac{4}{n-2}}w_{\ks}\varphi\right)+n T_c \int_{\d\Rn}u_{\ks}^{\frac{2}{n-2}}w_{\ks}\varphi\no\\
=&-\int_{\Rn}(B_1+B_2)\varphi
\end{align*}
for all $\varphi\in \mathcal{E}_{\ks}$. Since $w_{\ks}\in\mathcal{E}_{\ks}$, we obtain 
\begin{align*}
w_{\ks}=-G_{\ks}(B_1+B_2,0).
\end{align*}
By definitions of $v_{\ks}$ and $u_{\ks}$ we have 
\begin{align*}
v_{\ks}-u_{\ks}=G_{\ks}(B_3+B_4,A_1),
\end{align*}
where
\begin{align*}
B_3=&\Delta_g u_{\ks}-c_n R_g u_{\ks}+n(n-2)u_{\ks}^{\frac{n+2}{n-2}}\\
B_4=&n(n-2)|v_{\ks}|^{\frac{4}{n-2}}v_{\ks}-n(n-2)u_{\ks}^{\frac{n+2}{n-2}}-n(n+2)u_{\ks}^{\frac{4}{n-2}}(v_{\ks}-u_{\ks}),
\end{align*}
and
$$
A_1=-(n-2)T_c|v_{\ks}|^{\frac{2}{n-2}}v_{\ks}+(n-2)T_cu_{\ks}^{\frac{n}{n-2}}+n T_c u_{\ks}^{\frac{2}{n-2}}(v_{\ks}-u_{\ks}).
$$
Thus, we obtain 
\begin{align*}
v_{\ks}-u_{\ks}-w_{\ks}=G_{\ks}(B_1+B_2+B_3+B_4,A_1).
\end{align*}
It follows from \eqref{annihilator esti} that
\begin{align*}
&\|v_{\ks}-u_{\ks}-w_{\ks}\|_{\mathcal{E}}\\
\leq& C\|B_1+B_2+B_3+B_4\|_{L^{\frac{2n}{n+2}}(\Rn)}+C\|A_1\|_{L^{\frac{2(n-1)}{n}}(\d\Rn)}.
\end{align*}
Following the same lines in \cite[Corollary 8]{Brendle-Marques} and \cite[Proposition 7]{Brendle3}, together with Proposition \ref{prop:est_w} and \eqref{asym esti2} we obtain
\begin{align*}
\|B_1\|_{L^{\frac{2n}{n+2}}(\Rn)}\leq& C\lambda^{\frac{(2d+2)(n+2)}{n-2}}\mu^2 +C\rho^{2d+2}\mu\lambda^{\frac{n-2}{2}}\rho^{\frac{2-n}{2}},\\
\|B_2+B_3\|_{L^{\frac{2n}{n+2}}(\Rn)}\leq& C\lambda^{\frac{(2d+2)(n+2)}{n-2}}\mu^2+C\lambda^{\frac{n-2}{2}}\rho^{\frac{2-n}{2}}
\end{align*}
and by \eqref{asym esti1}, \eqref{sol esti},
\begin{align*}
\|B_4\|_{L^{\frac{2n}{n+2}}(\Rn)}\leq&\|v_{\ks}-u_{\ks}\|_{L^\frac{2n}{n-2}(\Rn)}^{\frac{n+2}{n-2}}\\
\leq& C\lambda^{\frac{(2d+2)(n+2)}{n-2}}\mu^{\frac{n+2}{n-2}}+C\lambda^{\frac{n+2}{2}}\rho^{\frac{-2-n}{2}},\\
\|A_1\|_{L^{\frac{2(n-1)}{n}}(\d\Rn)}\leq& C\|v_{\ks}-u_{\ks}\|_{L^{\frac{2(n-1)}{n-2}}(\d\Rn)}^{\frac{n}{n-2}}\\
\leq& C\lambda^{\frac{(2d+2)n}{n-2}}\mu^{\frac{n}{n-2}}+C\lambda^{\frac{n}{2}}\rho^{\frac{-n}{2}}.
\end{align*}
Therefore, putting these facts together, we obtain the desired estimate.
 \end{proof}
 
 \begin{proposition}\label{prop:A1}
 There holds
 \begin{align*}
  &\left|\int_{\Rn}\left(|\nabla v_{\ks}|_{g}^{2}- |\nabla u_{\ks}|_{g}^{2}+c_n R_g (v_{\ks}^{2}-u_{\ks}^{2})\right)\right.\\
  &+n(n-2)\int_{\Rn}\left[(|v_{\ks}|^{\frac{4}{n-2}}-u_{\ks}^{\frac{4}{n-2}})u_{\ks}v_{\ks}-(|v_{\ks}|^{\frac{2n}{n-2}}-u_{\ks}^{\frac{2n}{n-2}})\right]\\
  &\left. -\int_{\Rn}\displaystyle{\sum_{a,b=1}^{n}}\mu \lambda^{2d}f(\lambda^{-2}|x|^2)H_{ab}(x)\d_a \d_b u_{\ks} w_{\ks}\right.\\
 & -(n-2)T_c\int_{\d\Rn}(|v_{\ks}|^{\frac{2}{n-2}}-u_{\ks}^{\frac{2}{n-2}})u_{\ks}v_{\ks}\\
 &\left.+(n-2)T_c\int_{\d\Rn}(|v_{\ks}|^{\frac{2(n-1)}{n-2}}-u_{\ks}^{\frac{2(n-1)}{n-2}}) \right|\\
 \leq& C \lambda^{\frac{(4d+4)(n-1)}{n-2}}\mu^{\frac{2(n-1)}{n-2}}+C\mu\lambda^{2d+2+\frac{n-2}{2}}\rho^{\frac{2-n}{2}}+C\lambda^{n-2}\rho^{2-n}
 \end{align*}
 for $\ks\in \lambda\Omega$.
 \end{proposition}
 
 \begin{proof}
 By definition of $v_{\ks}$, we have 
 \begin{align*}
 0=&\int_{\Rn}\left[\langle\nabla v_{\ks},\nabla(v_{\ks}-u_{\ks})\rangle_g+c_n R_g v_{\ks}(v_{\ks}-u_{\ks}) \right]\\
 &-n(n-2)\int_{\Rn}|v_{\ks}|^{\frac{4}{n-2}}v_{\ks}(v_{\ks}-u_{\ks})\\
 &+(n-2)T_c\int_{\d\Rn}|v_{\ks}|^{\frac{2}{n-2}}v_{\ks}(v_{\ks}-u_{\ks}).
 \end{align*}
By \eqref{bubble_half_space}, \eqref{asym esti2} and \eqref{sol esti}, an integration by parts gives 
 \begin{align*}
 &\left|\int_{\Rn}\left[\langle \nabla u_{\ks}, \nabla v_{\ks}\rangle_g-|\nabla u_{\ks}|_{g}^{2} +c_n R_g u_{\ks}(v_{\ks}-u_{\ks})\right]\right.\\
 &~-\int_{\Rn}n(n-2)u_{\ks}^{\frac{n+2}{n-2}}(v_{\ks}-u_{\ks})\\
 &~+\left.\int_{\d\Rn}(n-2)T_c u_{\ks}^{\frac{n}{n-2}}(v_{\ks}-u_{\ks})\right.\\
 &~\left.-\int_{\Rn}\displaystyle{\sum_{a,b=1}^{n}}\mu\lambda^{2d}f(\lambda^{-2}|x|^2)H_{ab}(x)\d_a\d_b u_{\ks}(v_{\ks}-u_{\ks})\right|\\
 =&\left|\int_{\Rn}(\Delta_g u_{\ks} -c_n R_g u_{\ks}+n(n-2) u_{\ks}^{\frac{n+2}{n-2}}\right.\\
 &\qquad\left.+\sum_{a,b=1}^{n}\mu\lambda^{2d}f(\lambda^{-2}|x|^2)H_{ab}(x)\d_a\d_b u_{\ks})(u_{\ks}-v_{\ks})\right|\\
 \leq& \|v_{\ks}-u_{\ks}\|_{L^{\frac{2n}{n-2}}(\Rn)}\cdot\left \|\Delta_g u_{\ks} -c_n R_g u_{\ks}+n(n-2) u_{\ks}^{\frac{n+2}{n-2}}\right.\\
 &\qquad\qquad\qquad~~\left.+\sum_{a,b=1}^{n}\mu\lambda^{2d}f(\lambda^{-2}|x|^2)H_{ab}(x)\d_a\d_b u_{\ks}\right\|_{L^{\frac{2n}{n+2}}(\Rn)}\\
 \leq& C\lambda^{\frac{2n(2d+2)}{n-2}}\mu^3 +C\lambda^{2d+2}\mu \lambda^{\frac{n-2}{2}}\rho^{\frac{2-n}{2}}+C\lambda^{n-2}\rho^{2-n}.
 \end{align*}
 On the other hand, by Proposition \ref{prop:est_v-u-w} we have 
 \begin{align*}
 &\left| \int_{\Rn}\sum_{a,b=1}^{n}\mu\lambda^{2d}f(\lambda^{-2}|x|^2)H_{ab}(x)\d_a\d_b u_{\ks}(v_{\ks}-u_{\ks}-w_{\ks})\right|\\
 \leq& C\lambda^{2d+2}\mu\|v_{\ks}-u_{\ks}-w_{\ks}\|_{L^{\frac{2n}{n-2}}(\Rn)}\\
 \leq& C \lambda^{\frac{(4d+4)(n-1)}{n-2}}\mu^{\frac{2(n-1)}{n-2}}+C\lambda^{2d+2+\frac{n-2}{2}}\mu\rho^{\frac{2-n}{2}}+C\lambda^{n-2}\rho^{2-n}.
 \end{align*}
 Putting these facts together, we obtain the desired estimate.
  \end{proof}
 \begin{proposition}\label{prop:A2}
 There hold
  \begin{align*}
& \left|\int_{\Rn}\left(|v_{\ks}|^{\frac{4}{n-2}}-u_{\ks}^{\frac{4}{n-2}}\right)u_{\ks}v_{\ks}-\frac{2}{n}\int_{\Rn}\left(|v_{\ks}|^{\frac{2n}{n-2}}-u_{\ks}^{\frac{2n}{n-2}}\right) \right|\\
\leq& C\mu^{\frac{2n}{n-2}}\lambda^{\frac{(4d+4)n}{n-2}}+C \lambda^{n}\rho^{-n}
 \end{align*}
 and
 \begin{align*}
& \left|\int_{\d\Rn}\left(|v_{\ks}|^{\frac{2}{n-2}}-u_{\ks}^{\frac{2}{n-2}}\right)u_{\ks}v_{\ks}-\frac{1}{n-1}\int_{\d\Rn}\left(|v_{\ks}|^{\frac{2(n-1)}{n-2}}-u_{\ks}^{\frac{2(n-1)}{n-2}}\right) \right|\\
\leq& C\mu^{\frac{2(n-1)}{n-2}}\lambda^{\frac{(4d+4)(n-1)}{n-2}}+C \lambda^{n-1}\rho^{1-n}
 \end{align*}
  for $\ks\in \lambda\Omega$.
 \end{proposition}
 \begin{proof}
We only need to prove the second assertion, since the first one is similar to \cite[Proposition 12]{Brendle3} together with \eqref{sol esti}. Observe that
\begin{align*}
&\left|\left(|v_{\ks}|^{\frac{2}{n-2}}-u_{\ks}^{\frac{2}{n-2}}\right)u_{\ks}v_{\ks}-\frac{1}{n-1}\left(|v_{\ks}|^{\frac{2(n-1)}{n-2}}-u_{\ks}^{\frac{2(n-1)}{n-2}}\right) \right|\\
\leq& C|v_{\ks}-u_{\ks}|^{\frac{2(n-1)}{n-2}}.
\end{align*}
This together with \eqref{sol esti} implies
\begin{align*}
& \left|\int_{\d\Rn}\left(|v_{\ks}|^{\frac{2}{n-2}}-u_{\ks}^{\frac{2}{n-2}}\right)u_{\ks}v_{\ks}-\frac{1}{n-1}\int_{\d\Rn}\left(|v_{\ks}|^{\frac{2(n-1)}{n-2}}-u_{\ks}^{\frac{2(n-1)}{n-2}}\right) \right|\\
\leq& C\|v_{\ks}-u_{\ks}\|_{L^{\frac{2(n-1)}{n-2}}(\d\Rn)}^{\frac{2(n-1)}{n-2}}\leq C\mu^{\frac{2(n-1)}{n-2}}\lambda^{\frac{(4d+4)(n-1)}{n-2}}+C\lambda^{n-1}\rho^{1-n}.
\end{align*}
This proves the assertion.
 \end{proof}
 
\begin{proposition}\label{prop:A3}
There holds
\begin{align}\label{est:prop}
&\left|\int_{\Rn}\left( |\nabla u_{\ks}|_{g}^{2}+c_n R_g u_{\ks}^{2}-n(n-2)u_{\ks}^{\frac{2n}{n-2}}\right)+(n-2)T_c\int_{\d\Rn} u_{\ks}^{\frac{2(n-1)}{n-2}}\right.\no\\
&~~\left. -\frac{1}{2}\int_{B_{\rho}^+(0)}\sum_{a,b,c=1}^{n}h_{ac}h_{bc}\d_a u_{\ks}\d_b u_{\ks}+\frac{c_n}{4}\int_{B_{\rho}^+(0)}\sum_{a,b,c=1}^{n}(\d_c h_{ab})^2u_{\ks}^{2}\right|\no\\
 \leq& C\rho^{2d+2}\lambda^{4d+4}\mu^3+C\lambda^{n-2}\rho^{2-n}
\end{align}
 for $\ks\in \lambda\Omega$.
\end{proposition}

\begin{proof}
By equation \eqref{bubble_half_space} of $u_{\ks}$, we have 
 \begin{align*}
 \int_{\Rn}|\nabla u_{\ks}|^2 =n(n-2)\int_{\Rn}u_{\ks}^{\frac{2n}{n-2}}-(n-2)T_c\int_{\d\Rn}u_{\ks}^{\frac{2(n-1)}{n-2}}.
 \end{align*}
Then the LHS of \eqref{est:prop} becomes
 \begin{align*}
 &\left|\int_{\Rn}\left( |\nabla u_{\ks}|_{g}^{2}+c_n R_g u_{\ks}^{2}-|\nabla u_{\ks}|^2\right)-\frac{1}{2}\int_{B_{\rho}^+(0)}\sum_{a,b,c=1}^{n}h_{ac}h_{bc}\d_a u_{\ks}\d_b u_{\ks}\right.\\
&\left. +\frac{c_n}{4}\int_{B_{\rho}^+(0)}\sum_{a,b,c=1}^{n}(\d_c h_{ab})^2u_{\ks}^{2}\right|.
 \end{align*}
 Notice that 
 \begin{align*}
 &\left|g^{ab}(x)-\delta_{ab}(x)+h_{ab}(x)-\frac{1}{2}\sum_{a,b,c=1}^{n}h_{ac}(x)h_{bc}(x) \right|\\
 \leq& C|h(x)|^3\leq C\mu^3 (\lambda+|x|)^{6d+6}
 \end{align*}
for $x \in B_\rho^+(0)$.
 This implies 
 \begin{align*}
  &\left|\int_{\Rn}\left( |\nabla u_{\ks}|_{g}^{2}-|\nabla u_{\ks}|^2\right)+\int_{\Rn}\sum_{a,b=1}^{n}h_{ab}\d_a u_{\ks}\d_b u_{\ks}\right.\\
&\left.~~ -\frac{1}{2}\int_{B_{\rho}^+(0)}\sum_{a,b,c=1}^{n}h_{ac}h_{bc}\d_a u_{\ks}\d_b u_{\ks}\right|\\
&\leq C\lambda^{n-2}\mu^3\int_{B_{\rho}^+(0)}(\lambda+|x|)^{6d+6+2-2n} +C\lambda^{n-2}\int_{\Rn\setminus B_{\rho}^+(0)}(\lambda+|x|)^{2-2n}\\
&\leq C\mu^3 \rho^{2d+2}\lambda^{4d+4}+C\lambda^{n-2}\rho^{2-n}.
 \end{align*}
Since $\pa_a h_{ab}(x)=0$ in $B_\rho^+(0)$, it follows from \cite[Proposition 4]{Brendle-Marques} that 
 \begin{align*}
  &\left|R_g(x)+\frac{1}{4}\sum_{a,b,c=1}^{n}(\d_c h_{ab}(x))^2 \right|\\
  \leq& C|h(x)|^2 |\d^2 h(x)| +C |h(x)||\d h(x)|^2\leq C\mu^3 (\lambda+|x|)^{6d+4}
 \end{align*}
 for $x \in B_\rho^+(0)$. This implies 
 \begin{align*}
  &\left|\int_{\Rn} R_g(x)u_{\ks}^{2}+\frac{1}{4}\int_{B_{\rho}^+(0)}\sum_{i,k,l=1}^{n}(\d_l h_{ik}(x))^2 u_{\ks}^2 \right|\\
  \leq& C \lambda^{n-2}\mu^3 \int_{B_{\rho}^+(0)}(\lambda+|x|)^{6d+6+2-2n}+C\lambda^{n-2}\int_{\Rn\setminus B_{\rho}^+(0)}(\lambda+|x|)^{4-2n}\\
 \leq& C\mu^3 \rho^{2d+2}\lambda^{4d+4}+C\lambda^{n-2}\rho^{4-n} .
 \end{align*}
Since $h_{ab}(x)$ is trace-free, by \eqref{eq:Einstein} we obtain 
 \begin{align*}
 \sum_{a,b=1}^{n}h_{ab}\d_a u_{\ks}\d_b u_{\ks} =c_n\sum_{a,b=1}^{n}h_{ab}\d_a\d_b u_{\ks}^2 ,
 \end{align*}
 then
 \begin{align*}
 \int_{\Rn} \sum_{a,b=1}^{n}h_{ab}\d_a u_{\ks}\d_b u_{\ks} =c_n\int_{\Rn}\sum_{a,b=1}^{n}h_{ab}\d_a\d_b u_{\ks}^2 .
 \end{align*}
 Again by $\pa_a h_{ab}(x)=0$ in $B_\rho^+(0)$, we obtain
 \begin{align*}
 \left| \int_{\Rn}\sum_{a,b=1}^{n}h_{ab}\d_a u_{\ks}\d_b u_{\ks}\right|\leq \int_{\Rn\setminus B_{\rho}^+(0)}u_{\ks}^{2}\leq C\lambda^{n-2}\rho^{4-n}.
 \end{align*}
 Then the desired estimate follows from all the above facts.
 \end{proof}

Consequently, collecting Propositions \ref{prop:A1}- \ref{prop:A3} together, we arrive at the following key estimate.
 \begin{corollary}\label{Cor:key_est}
 Let $\mathcal{F}_{g}(\xi,\e)$ be the function defined in \eqref{energy functional}, then for any $\ks\in\lambda\Omega$, there holds
 \begin{align*}
 &\left|\mathcal{F}_{g}(\xi,\e)-\frac{1}{2}\int_{B_{\rho}^{+}(0)}\sum_{a,b,c=1}^n h_{ac}h_{bc}\d_a u_{\ks}\d_{b} u_{\ks} +\frac{c_n}{4} \int_{B_{\rho}^{+}(0)}\sum_{a,b,c=1}^n(\d_c h_{ab})^{2}u_{\ks}^{2}\right.\\
 &\left.-\mu\lambda^{2d}\int_{\Rn}f(\lambda^{-2}|x'|^{2})w_{\ks}\sum_{a,b=1}^n H_{ab}(x)\d_a\d_b u_{\ks}\right|\no\\
 \leq& C\mu^{\frac{2(n-1)}{n-2}}\lambda^{\frac{(4d+4)(n-1)}{n-2}}+C\mu\lambda^{2d+2+\frac{n-2}{2}}\rho^{\frac{2-n}{2}}+C\lambda^{n-2}\rho^{2-n},
 \end{align*}
 where $w_{\ks}$ satisfies \eqref{linear approximation}.
 \end{corollary}
 
 \section{Finding a critical point of an auxiliary function}\label{Sect5}
We define
 \begin{align*}
 \mathcal{F}\ks=&\frac{1}{2}\int_{\Rn}\sum_{i,j,l=1}^{n-1}\overline H_{il}\overline H_{jl}\d_i u_{\ks}\d_{j} u_{\ks} -\frac{c_n}{4} \int_{\Rn}\sum_{i,j,l=1}^{n-1}(\d_l \overline H_{ij})^{2}u_{\ks}^{2}\\
 &+\int_{\Rn}\sum_{i,j=1}^{n-1} \overline H_{ij}\d_i\d_j u_{\ks}z_{\ks}
 \end{align*}
 for $\ks \in \Rn\times (0,\infty)$, where $z_{\ks}(x)=\mu^{-1} w_{\ks}(x)$, which satisfies
 \begin{align}
  &\int_{\Rn}\langle \nabla z_{\ks},\nabla\varphi\rangle-n(n+2)u_{\ks}^{\frac{4}{n-2}}z_{\ks}\varphi+n T_c\int_{\d\Rn} u_{\ks}^{\frac{2}{n-2}}z_{\ks}\varphi\no\\
 &=-\int_{\Rn}f(|x'|^2)\overline H_{ij}(x)\d_i\d_ju_{\ks}\varphi
 \end{align}
 for all test function $\varphi\in \mathcal{E}_{\ks}$.
  
Next we show that the function $\mathcal F{\ks}$ has a strict local minimum. Throughout this section we use indices $1\leq i,j,k,l,m,p,q,r\leq n-1$ .

Since $\overline H_{ab}(-x)=\overline H_{ab}(x)$ for any $x\in\Rn$, the function $\mathcal{F}\ks$ satisfies $\mathcal{F}\ks=\mathcal F(-\xi,\epsilon)$ for all $\ks\in\R^{n-1}\times(0,\infty)$. This implies
 \begin{equation}\label{eq:mixed_d^2F}
 \frac{\d}{\d\xi_p}\mathcal F(0,\epsilon)=\frac{\d^{2}}{\d\epsilon\d\xi_p}\mathcal F(0,\epsilon)=0
 \end{equation}
  for all $\epsilon>0$ .
  \begin{proposition}\label{prop:elementary_iden1}
  There hold
  \begin{align*}
   &\int_{\Sp^{n-2}_{r}(0)}\displaystyle{\sum_{i,k,l=1}^{n-1}}(\d_l H_{ik}(x))^2 x_p x_q\\
   =&\frac{2r^{n+2}}{(n-1)(n+1)}|\Sp^{n-2}|\displaystyle{\sum_{i,k,l=1}^{n-1}}(\overline{W}_{ipkl}+\overline W_{ilkp})(\overline W_{iqkl}+\overline W_{ilkq})\\
   &+\frac{r^{n+2}}{(n-1)(n+1)}|\Sp^{n-2}|\displaystyle{\sum_{i,j,k,l=1}^{n-1}}(\overline W_{ijkl}+\overline W_{ilkj})^2 \delta_{pq}
  \end{align*}
  and 
  \begin{align*}
  &\int_{\Sp^{n-2}_{r}(0)}\displaystyle{\sum_{i,k=1}^{n-1}}H_{ik}(x)^2 x_p x_q \\
  =&\frac{2r^{n+4}}{(n-1)(n+1)(n+3)}|\Sp^{n-2}|\displaystyle{\sum_{i,k,l=1}^{n-1}}(\overline W_{ipkl}+\overline W_{ilkp})(\overline W_{iqkl}+\overline W_{ilkq})\\
  &+\frac{r^{n+4}}{2(n-1)(n+1)(n+3)}|\Sp^{n-2}|\displaystyle{\sum_{i,j,k,l=1}^{n-1}}(\overline W_{ijkl}+\overline W_{ilkj})^2\delta_{pq}.
  \end{align*}
  \end{proposition}
  \begin{proof}
The proof is similar to \cite[Proposition 16]{Brendle3}.
  \end{proof}
  \begin{proposition}\label{cal1}
 There holds
  \begin{align*}
  &\int_{\Sp_{r}^{n-2}(0)}\sum_{i,k,l=1}^{n-1}(\d_{l}\overline H_{ik}(x))^2x_p x_q\\
 =&\left\{2(\overline W_{ipkl}+\overline W_{ilkp})(\overline W_{iqkl}+\overline W_{ilkq})\left[ (n+3)f(r^2)^2+8r^2f(r^2)f^\prime(r^2)+4r^4f^\prime(r^2)^2 \right]\right.\\
 &\left.~+(\overline W_{ijkl}+\overline W_{ilkj})^2\delta_{pq}\left[ (n+3)f(r^2)^2+4r^2f(r^2)f^\prime(r^2)+2r^4f^\prime(r^2)^2\right]\right\}\\
 &\cdot\frac{|\Sp^{n-2}|r^{n+2}}{(n-1)(n+1)(n+3)}.
  \end{align*} 
 \end{proposition}
   \begin{proof}
   Since 
   \begin{align*}
   \d_l \overline H_{ik}(x)=f(|x'|^2)\d_l H_{ik}(x)+2x_l f'(|x'|^2)H_{ik}(x),
   \end{align*}
and by Euler's formula we obtain 
 \begin{align*}
 &\sum_{i,k,l=1}^{n-1}(\d_{l}\overline H_{ik}(x))^2\\
=&f(|x'|^2)^2\displaystyle{\sum_{i,k,l=1}^{n-1}}(\d_l H_{ik}(x))^2+4f(|x'|^2)f'(|x'|^2)\displaystyle{\sum_{i,k,l=1}^{n-1}}H_{ik}(x)x_l\d_l H_{ik}(x)\\
&+4|x'|^2 f'(|x'|^2)^2\sum_{i,k=1}^{n-1}H_{ik}(x)^{2}\\
=&f(|x'|^2)^2\displaystyle{\sum_{i,k,l=1}^{n-1}}(\d_l H_{ik}(x))^2+4\left[2f(|x'|^2)f'(|x'|^2)+|x'|^2 f'(|x'|^2)^2\right]\sum_{i,k=1}^{n-1}H_{ik}(x)^{2}.
 \end{align*}
 Hence, the assertion follows from Proposition \ref{prop:elementary_iden1}.
  \end{proof}
  \begin{corollary}\label{cal2}
  There holds
  \begin{align*}
   &\int_{\Sp_{r}^{n-2}(0)}\sum_{i,k,l=1}^{n-1}(\d_l \overline H_{ik})^2 (x)\\
    =&\frac{|\Sp^{n-2}|r^n}{(n-1)(n+1)}\sum_{i,j,k,l=1}^{n-1}(\overline W_{ijkl}+\overline W_{ilkj})^2 \left[(n+1)f(r^2)^2 +4r^2 f(r^2) f'(r^2)+2r^4f'(r^2)^2 \right].
  \end{align*}
   \end{corollary}
   
  \begin{proposition}\label{prop:F}
  There holds
  \begin{align*}
  &\mathcal{F}(0,\epsilon)\\
  =&-\frac{c_n |\Sp^{n-2}|\epsilon^{n-2}}{4(n-1)(n+1)}\sum_{i,j,k,l=1}^{n-1}(\overline W_{ijkl}+\overline W_{ilkj})^2\int_{0}^{\infty}\int_{0}^{\infty}(\epsilon^2+(t-T_c\epsilon)^2+r^2)^{2-n} \\
  &\cdot r^n\left[ (n+1)f(r^2)^2 +4r^2 f(r^2) f'(r^2)+2r^4f'(r^2)^2\right]\ud r \ud t.
  \end{align*}
  \end{proposition}
  \begin{proof}
Since $\sum_{i,j=1}^{n-1}\bar H_{ij}(x)\d_i \d_j u_{(0,\epsilon)}(x)=0$,
 then $z_{(0,\epsilon)}=0$, and by symmetry we have 
  \begin{align*}
 \int_{\Sp_{r}^{n-2}(0)}\sum_{i,k,l=1}^{n-1}\overline H_{il}(x)\overline H_{kl}(x)\d_i u_{(0,\epsilon)}(x)\d_k u_{(0,\epsilon)}(x)=0.
  \end{align*}
Then we have 
  \begin{align*}
 \mathcal{F}(0,\epsilon)=&-\frac{c_n}{4}\int_{\Rn}(\d_l \overline H_{ik})^2(x)u_{(0,\epsilon)}^2 (x)\\
  &=-\frac{c_n}{4}\int_{0}^{\infty}\int_{0}^{\infty}\int_{\Sp_{r}^{n-2}(0)}(\d_l \overline H_{ik})^2(x) u_{(0,\epsilon)}^{2}(x) \ud\sigma_r (x) \ud r \ud x_n . 
  \end{align*}
Hence, the result follows from Corollary \ref{cal2}.
  \end{proof}
 By Proposition \ref{prop:F}, we rewrite 
 \begin{align*}
 \mathcal{F}(0,\epsilon)=&-\frac{(n-2) |\Sp^{n-2}|}{16(n-1)^2(n+1)}(\overline W_{ijkl}+\overline W_{ilkj})^2\\
 &\cdot\sum_{q=0}^{2d}\alpha_q \int_{0}^{\infty}\int_{0}^{\infty}r^{2q+n}\epsilon^{n-2}(\epsilon^2 +(t-T_c \epsilon)^2 +r^2)^{2-n} \ud r\ud t,
 \end{align*}
 where $\alpha_q $ are constants defined by
 \begin{align}\label{eq:alpha}
\sum_{q=0}^{2d}\alpha_q s^q:=(n+1)f(s)^2 +4s f(s)f'(s)+2s^2 f'(s)^2 .
 \end{align}
 Then we obtain
 \begin{align*}
 \mathcal{F}(0,\epsilon)=&-\frac{(n-2) |\Sp^{n-2}|}{16(n-1)^2(n+1)}(\overline W_{ijkl}+\overline W_{ilkj})^2\\
 &\cdot \sum_{q=0}^{2d}\alpha_q \epsilon^{2q+4}\int_{0}^{\infty}\int_{0}^{\infty}\frac{r^{2q+n}}{(1+(t-T_c)^2 +r^2)^{n-2}}\ud r\ud t,\\
=&-\frac{(n-2) |\Sp^{n-2}|}{16(n-1)^2(n+1)}(\overline W_{ijkl}+\overline W_{ilkj})^2\\
 &\cdot \sum_{q=0}^{2d}\alpha_q \epsilon^{2q+4}\int_{0}^{\infty}\frac{1}{(1+(t-T_c)^2)^{\frac{n-5-2q}{2}}}\ud t\int_{0}^{\infty}\frac{r^{2q+n}}{(1+r^2)^{n-2}} \ud r\\
 =&-\frac{(n-2) |\Sp^{n-2}|}{16(n-1)^2(n+1)}(\overline W_{ijkl}+\overline W_{ilkj})^2\sum_{q=0}^{2d}\alpha_q c_q \epsilon^{2q+4}  B(\frac{2q+n+1}{2},\frac{n-5-2q}{2}).
 \end{align*}
 where
 \begin{equation}\label{def:c_q}
 c_q=\int_{0}^{\infty}(1+(t-T_c)^2)^{\frac{5+2q-n}{2}}\ud t, \quad q \in \mathbb{N} \mathrm{~~and~~} 0 \leq q \leq 2d.
 \end{equation}
 For clarity, we rewrite 
 \begin{align}\label{eq:F(0,epsilon)}
 \mathcal{F}(0,\epsilon)=-\frac{(n-2) |\Sp^{n-2}|}{32(n-1)^2(n+1)}B(\frac{n-1}{2},\frac{n-5}{2})(\overline W_{ijkl}+\overline W_{ilkj})^2I(\epsilon^2),
 \end{align}
 where
 \begin{align}\label{def:I}
 I(s)=\sum_{q=0}^{2d}\left(c_q \alpha_q s^{q+2} \prod_{j=0}^{q}\frac{n-1+2j}{n-5-2j}\right).
 \end{align}
 \begin{proposition}\label{prop:d_xi^2 F}
There holds
\begin{align*}
 &\frac{\d^2}{\d\xi_p\d \xi_q}F(0,\epsilon)\\
 =&-\frac{2(n-2)^2|\Sp^{n-2}|}{(n-1)(n+1)(n+3)}(\overline W_{ipkl}+\overline W_{ilkp})(\overline W_{iqkl}+\overline W_{ilkq})\\
 &\cdot\int_{0}^{\infty}\int_{0}^{\infty}\frac{\epsilon^{n-2}}{(\epsilon^2 +(t-T_c \epsilon)^2 +r^2)^n}r^{n+4}(2f(r^2)f'(r^2)+r^2 f'(r^2)^2)\ud r\ud t\\
 &-\frac{(n-2)^2|\Sp^{n-2}|}{2(n-1)(n+1)(n+3)}(\overline W_{ijkl}+\overline W_{ilkj})^2 \delta_{pq}\\
 &\cdot \int_{0}^{\infty}\int_{0}^{\infty}\frac{\epsilon^{n-2}}{(\epsilon^2 +(t-T_c \epsilon)^2 +r^2)^n}r^{n+4}(2f(r^2)f'(r^2)+r^2 f'(r^2)^2)\ud r\ud t\\
 &+\frac{(n-2)^2|\Sp^{n-2}|}{4(n-1)^2 (n+1)}(\overline W_{ijkl}+\overline W_{ilkj})^2 \delta_{pq}\\
&\cdot \int_{0}^{\infty}\int_{0}^{\infty}\frac{\epsilon^{n-2}}{(\epsilon^2 +(t-T_c \epsilon)^2 +r^2)^{n-1}}r^{n+4}f'(r^2)^2 \ud r\ud t.
\end{align*}
 \end{proposition}
 \begin{proof}
 As in \cite[Proposition 21]{Brendle3}, similarly we obtain
 \begin{align*}
 &\frac{\d^2}{\d\xi_p\d \xi_q}\mathcal F(0,\epsilon)\\
 =&(n-2)^2 \int_{\Rn}\frac{\epsilon^{n-2}}{(\epsilon^2 +(x_n-T_c \epsilon)^2 +|x'|^2)^n}\overline H_{pl}(x)\overline H_{ql}(x)\\
 &-\frac{(n-2)^2}{4}\int_{\Rn}\frac{\epsilon^{n-2}}{(\epsilon^2 +(x_n-T_c \epsilon)^2 +|x'|^2)^n}(\d_l \overline H_{ik}(x))^2 x^p x^q \\
 &+\frac{(n-2)^2}{8(n-1)}\int_{\Rn}\frac{\epsilon^{n-2}}{(\epsilon^2 +(x_n-T_c \epsilon)^2 +|x'|^2)^{n-1}}(\d_l \overline H_{ik}(x))^2 \delta_{pq}.
 \end{align*}
 This together with Proposition \ref{cal1} and Corollary \ref{cal2} gives the desired assertion.
 \end{proof}
 For brevity, we let
 \begin{align}\label{def:beta_q}
 2f(s) f'(s) +s f'(s)^2:=\sum_{q=0}^{2d-1}\beta_q s^q.
 \end{align}
By definition \eqref{def:c_q} of $c_q$, a direct computation yields
 \begin{align*}
 &\int_{0}^{\infty}\int_{0}^{\infty}\frac{\epsilon^{n-2}}{(\epsilon^2 +(t-T_c \epsilon)^2 +r^2)^n}r^{n+4}(2f(r^2)f'(r^2)+r^2 f'(r^2)^2)\ud r \ud t\\
 :=&\frac{1}{2}B(\frac{n+3}{2},\frac{n-3}{2})J(\epsilon^2),
 \end{align*}
 where
 \begin{align}\label{def:J}
J(s)=&\sum_{q=0}^{2d-1}\beta_q c_q s^{q+2} \frac{B(\frac{n+2q+5}{2},\frac{n-2q-5}{2})}{B(\frac{n+3}{2},\frac{n-3}{2})}\no\\
=&\sum_{q=0}^{2d-1}\left(c_q\beta_q s^{q+2}\prod_{j=0}^{q}\frac{n+3+2j}{n-5-2j}\right).
 \end{align}
 
In order to show that $\mathcal{F}(\xi,\e)$ has a strict local minimum at $(0,1)$, By \eqref{eq:mixed_d^2F},\eqref{eq:F(0,epsilon)} and Proposition \ref{prop:d_xi^2 F}, our strategy is to find some polynomials $f(s):=\sum_{i=0}^d a_i s^i, a_i \in \mathbb{R}$ for $1 \leq i \leq d$, such that $I(1)>0, I'(1)=0, I''(1)<0$ and $J(1)<0$. 

Before proceeding to find such polynomials $f$, we first need the following elementary result.
\begin{lemma}\label{numesti1}
Let $T_c<0, n \geq 25$ and $c_q$ be defined in \eqref{def:c_q}, there holds
$$(1+T_c^2)\frac{n-2q-7}{n-2q-8}\leq \frac{c_{q+1}}{c_q}\leq \frac{n-2q-6}{n-2q-8}(1+T_c^2)$$
for $0\leq q \leq 2d$. In particular,
\begin{align*}
\frac{(n-6)(n-10)}{(n-8)(n-9)}\geq \frac{c_{1}^{2}}{c_0 c_2}\geq \frac{(n-10)(n-7)}{(n-8)^2}.
\end{align*}
\end{lemma}
\begin{proof}
Let $a=-T_c>0$ and define
$$I_\alpha(a):=\int_a^\infty(1+r^2)^{-\alpha} dr \quad \mathrm{~~for~~} \alpha >\frac{1}{2}.$$
An integration by parts gives
\begin{align}\label{def:I_alpha}
I_\alpha(a)=\frac{2\alpha}{2\alpha-1}I_{\alpha+1}(a)+\frac{1}{2\alpha-1}(1+a^2)^{-\alpha} a.
\end{align}
Notice that
\begin{align*}
\lim_{k \to \infty}\prod_{i=0}^{k-1}\frac{2\alpha+2i}{2\alpha+2i-1}I_{\alpha+k}(a)\leq&\frac{\pi}{2} \lim_{k \to \infty}\prod_{i=0}^{k-1}\frac{2\alpha+2i}{2\alpha+2i-1}(1+a^2)^{-\alpha-k+1}\\
=&\frac{\pi}{2} \lim_{k \to \infty}\frac{\displaystyle \prod_{i=0}^{k-1}\left(1+\frac{1}{2\alpha+2i-1}\right)}{(1+a^2)^{\alpha-(k-1)}}=0.
\end{align*}
From this, we iterate \eqref{def:I_alpha} to obtain
\begin{align*}
I_\alpha(a)=a\sum_{k=0}^{\infty}\frac{1}{2\alpha+2k-1}\prod_{i=0}^{k-1}\frac{2\alpha+2i}{2\alpha+2i-1}(1+a^2)^{-(\alpha+k)}.
\end{align*}
Then we have 
\begin{align*}
&\frac{I_\alpha(a)}{(1+a^2)I_{\alpha+1}(a)}\\
=&\frac{\displaystyle{\sum_{k=0}^{\infty}}\frac{1}{2\alpha+2k-1}\displaystyle{\prod_{i=0}^{k-1}}\frac{2\alpha+2i}{2\alpha+2i-1}(1+a^2)^{-(\alpha+k)}}{\displaystyle{\sum_{k=0}^{\infty}}\frac{1}{2\alpha+2k+1}\displaystyle{\prod_{i=0}^{k-1}}\frac{2\alpha+2i+2}{2\alpha+2i+1}(1+a^2)^{-(\alpha+k)}} .
\end{align*}
Since
\begin{align*}
\frac{\displaystyle{\frac{1}{2\alpha+2k-1}}\displaystyle{\prod_{i=0}^{k-1}}\frac{2\alpha+2i}{2\alpha+2i-1}}{\displaystyle{\frac{1}{2\alpha+2k+1}}\displaystyle{\prod_{i=0}^{k-1}}\frac{2\alpha+2i+2}{2\alpha+2i+1}}=\frac{2\alpha(2\alpha+2k+1)}{(2\alpha-1)(2\alpha+2k)},
\end{align*}
then 
\begin{align*}
\frac{2\alpha}{2\alpha-1}\leq \frac{I_\alpha(a)}{(1+a^2)I_{\alpha+1}(a)}\leq \frac{2\alpha+1}{2\alpha-1}.
\end{align*}
In particular, it follows from \eqref{def:c_q} and the above inequality that for $0 \leq q \leq 2d$,
$$c_q=I_{\frac{n-2q-5}{2}}(a)  \quad \mathrm{and}$$
$$(1+T_c^2)\frac{n-2q-7}{n-2q-8}\leq \frac{c_{q+1}}{c_q}\leq \frac{n-2q-6}{n-2q-8}(1+T_c^2).$$
And the remained estimates follow from the above estimate.
\end{proof}

Now we choose $d=1$ and let $f(s)=a_0+a_1 s$. Then by \eqref{eq:alpha} we obtain
 \begin{align*}
 \alpha_0 =(n+1)a_{0}^{2},\quad \alpha_1 =2(n+3)a_0 a_1, \quad \alpha_{2}=(n+7)a_{1}^{2}.
 \end{align*}
Differentiating \eqref{def:I} with respect to $s$, we obtain
\begin{align*}
&I'(s)=\sum_{q=0}^{2}\left[(q+2)c_q \alpha_q s^{q+1}\prod_{j=0}^{q}\frac{n-1+2j}{n-5-2j} \right]\\
=&\frac{2 c_0 \alpha_0 (n-1)}{n-5}s+\frac{3c_1\alpha_1(n-1)(n+1)}{(n-5)(n-7)}s^2 +\frac{4 c_2 \alpha_2 (n-1)(n+1)(n+3)}{(n-5)(n-7)(n-9)}s^3\\
&=\frac{2(n+1)(n-1)}{n-5}\left[c_0 a_{0}^{2}s+\frac{3(n+3)}{n-7}c_1 a_0 a_1 s^2 +\frac{2(n+3)(n+7)}{(n-7)(n-9)}c_2 a_{1}^{2}s^3  \right].
\end{align*}
We set $a_1=-1$ and define
\begin{align*}
p_n(a_0):=c_0 a_{0}^{2} -\frac{3(n+3)}{n-7}c_1 a_0 +\frac{2(n+3)(n+7)}{(n-7)(n-9)}c_2,
\end{align*}
then
\begin{align*}
I'(1)=\frac{2(n+1)(n-1)}{n-5}p_n(a_0).
\end{align*}
Notice that the discriminant $d(p_n)$ of $p_n$ is given by 
\begin{align*}
d(p_n)=\frac{(n+3)^2}{(n-7)^2}\left[ 9c_{1}^{2}-8\frac{(n+7)(n-7)}{(n+3)(n-9)}c_0 c_2\right].
\end{align*}

By Lemma \ref{numesti1}, we have 
\begin{align*}
d(p_n)\geq\frac{(n+3)^2}{(n-7)^2}c_{1}^{2} \left[ 9-8\frac{(n+7)(n-8)^2}{(n+3)(n-9)(n-10)}\right].
\end{align*}
Define\begin{align*}
q(n) = 9(n+3)(n-9)(n-10)-8(n-8)^2(n+7),
\end{align*}
then 
\begin{align*}
q'(n)=3n^2 -144n+681.
\end{align*}
Notice that $q'(x)>0$ for $ x \geq 62$ and $q(62)>0, q(61)<0$, then $d(p_n)>0$ for $n\geq 62$.

Hence,  we can choose
\begin{align}\label{def:a_0}
a_0=\frac{(n+3)c_1}{2(n-7)c_0}\left[3+\sqrt{9-8\frac{(n+7)(n-7)}{(n+3)(n-9)}\frac{c_0 c_2}{c_{1}^{2}}}\right]
\end{align}
such that $I'(1)=0$. From this and Lemma \ref{numesti1}, we obtain
\begin{align*}
I(1)=&\frac{1}{3}\left[c_0a_0^2-\frac{(n+3)(n+7)}{(n-7)(n-9)}c_2\right]\\
\geq&\frac{c_2}{3}\frac{n+3}{n-7}\left[\frac{9(n+3)}{4(n-7)}\frac{c_1^2}{c_0c_2}-\frac{n+7}{n-9}\right]\\
\geq&\frac{c_2}{3}\frac{n+3}{n-7}\left[\frac{9(n+3)(n-10)}{4(n-8)^2}-\frac{n+7}{n-9}\right]>0
\end{align*}
for $n \geq 62$.

\begin{lemma}\label{numesti2}
There holds $I''(1)<0$ for $n\geq 62$.
\end{lemma}
\begin{proof}
By definition \eqref{def:I} of $I$, we have 
\begin{align*}
I''(1)=\frac{2(n+1)(n-1)}{n-5}\left(c_0 a_{0}^{2}-\frac{6(n+3)}{n-7}c_1 a_0 +\frac{6(n+3)(n+7)}{(n-7)(n-9)}c_2 \right).
\end{align*}
Notice that
\begin{align*}
I''(1)-I'(1)=\frac{2(n-1)(n+1)}{n-5}\left[-3c_1 a_0 \frac{(n+3)}{(n-7)} +4c_2 \frac{(n+3)(n+7)}{(n-7)(n-9)}\right].
\end{align*}
By definition \eqref{def:a_0} of $a_0$ and $I'(1)=0$, we have
\begin{align*}
I''(1)\leq& \frac{2(n-1)(n+1)}{n-5}\left[-\frac{9}{2}\frac{c_{1}^{2}}{c_0}\frac{(n+3)^2}{(n-7)^2}+4c_2 \frac{(n+3)(n+7)}{(n-7)(n-9)}\right]\\
=&-\frac{(n-1)(n+1)}{(n-5)c_0}d(p_n)<0
\end{align*}
for $n\geq 62$. This implies the desired result.  
\end{proof}

\begin{lemma}\label{numesti3}
There holds $J(1)<0$ for $n\geq 62$.
\end{lemma}
\begin{proof}
By definition \eqref{def:J} of $J$, letting $a_1=-1$ and $a_0$ be chosen as in \eqref{def:a_0} we obtain $\beta_0 =-2 a_0$, $\beta_1=3$ by virtue of \eqref{def:beta_q}, and
\begin{align*}
J(s)=-\frac{2(n+3)}{n-5}c_0 a_0 s^2 +\frac{3(n+3)(n+5)}{(n-5)(n-7)}c_1  s^3.
\end{align*}
Thus we have
\begin{align*}
J(1)=\frac{(n+3)c_1}{(n-5)(n-7)} \left[6-(n+3)\sqrt{9-8\frac{(n+7)(n-7)}{(n+3)(n-9)}\frac{c_0 c_2}{c_{1}^{2}}} \right].
\end{align*}
By Lemma \ref{numesti1} we have 
\begin{align*}
J(1)\leq \frac{(n+3)c_1}{(n-5)(n-7)} \left[6-(n+3)\sqrt{9-\frac{8(n+7)(n-8)^2}{(n+3)(n-9)(n-10)}}\right].
\end{align*}
We consider
\begin{align*}
 \mathcal{P}(n):=\alpha (n+3)(n-9)(n-10)-(n+7)(n-8)^2 \mathrm{~~with~~} \alpha=\frac{1}{8}\left(9-\frac{36}{65^2} \right).
\end{align*}
Then a direct computation shows that $\mathcal{P}''(n)=6(\alpha-1)n+18-32\alpha>0$ for $n\geq 62$ and $\mathcal{P}'(62)>0, \mathcal{P}(62)>0$. This implies that $\mathcal{P} (n) >0$~~for $n \geq 62$. Observe that
\begin{align*}
(n+3)\sqrt{9-8\frac{(n+7)(n-8)^2}{(n+3)(n-9)(n-10)}} > (n+3)\sqrt{9-8 \alpha}\geq 6,
\end{align*}
where the last inequality follows from $n\geq 62$ and the choice of $\alpha$. This yields $J(1)<0$ for $n \geq 62$.
\end{proof}

Combing the above facts with Lemmas \ref{numesti2}-\ref{numesti3} we arrive at 
\begin{proposition}\label{prop:n>61}
Let $n\geq 62$, then there exists a polynomial $f(s)=-s+a_0$ with 
\begin{align*}
a_0 =\frac{(n+3)c_1}{2(n-7)c_0}\left[3+\sqrt{9-8\frac{(n+7)(n-7)}{(n+3)(n-9)}\frac{c_0 c_2}{c_{1}^{2}}}\right]
\end{align*}
such that $I(1)>0, I'(1)=0$, $I''(1) <0$ and $J(1)<0$. This implies that $\mathcal{F}\ks$ has a strict local minimum at the point $(0,1)$. 
\end{proposition}

\section{Proof of Theorem \ref{Thm:main}}\label{Sect6}
\begin{proposition}\label{main1}
For $n\geq 62$, let $g=\exp(h)$ is a smooth Riemannian metric on $\overline{\Rn}$, where $h$ is a symmetric trace-free two tensor on $\overline{\Rn}$ satisfying
\begin{align*}
\begin{cases}
\displaystyle h_{ij}(x)=\mu\lambda^{2}f(\lambda^{-2}|x'|^2)H_{ij}(x), &\quad  \mathrm{in~~} B_\rho^+, \\
\displaystyle h_{ab}(x)=0, &\quad \mathrm{in~~} \overline \Rn\setminus B_\rho^+,\\
\displaystyle h_{na}(x)=0, &\quad \mathrm{in~~} \overline \Rn,\\
\end{cases}
\end{align*}
where $f(s)=a_0-s$ with the constant $a_0$ given in Proposition \ref{prop:n>61}, $0<\mu \leq 1, 0<\lambda\leq\rho\leq 1$, $1\leq i \leq n-1$, $1\leq a,b \leq n$ and $H_{ab}$ is defined in \eqref{def:H_ab}. Assume that $|h(x)|+|\d h(x)| +|\d^2 h(x)| \leq \alpha$ for all $x\in  \overline \Rn$. If $\alpha$ and $\mu^{-2}\lambda^{n-10}\rho^{2-n}$ are sufficiently small, then there exists a positive smooth solution of
\begin{align}\label{PDE:single_v}
\begin{cases}
\displaystyle \Delta_g v -c_n R_g v+n(n-2)v^{\frac{n+2}{n-2}}=0,&\quad \mathrm{in~~} \Rn,\\
\displaystyle \frac{\d v}{\d x_n}=(n-2)T_c v^{
\frac{n}{n-2}}, &\quad \mathrm{on~~} \d\Rn.
\end{cases}
\end{align}
Moreover, there exists $C=C(n,T_c)>0$, such that 
\begin{align*}
\sup_{B_{\lambda}^+(0)} v\geq C \lambda^{\frac{2-n}{2}}
\end{align*}
and 
\begin{align*}
&2(n-2)\int_{\Rn}v^{\frac{2n}{n-2}}-\frac{n-2}{n-1}T_c \int_{\d\Rn}v^{\frac{2(n-1)}{n-2}}\\
<&2(n-2)\int_{\Rn}u_{(0, 1)}^{\frac{2n}{n-2}}-\frac{n-2}{n-1}T_c\int_{\d\Rn}u_{(0,1)}^{\frac{2(n-1)}{n-2}}.
\end{align*}
\end{proposition}
\begin{proof}
It follows from Proposition \ref{prop:n>61} that $(0,1)$ is a strict local minimum point of $\mathcal{F}\ks$. Hence, we can find an open set $\Omega' \subset \Omega$ such that $(0,1)\in \Omega'$ and $\mathcal{F}(0,1)<\displaystyle{\inf_{\ks\in\d\Omega'}}\mathcal{F}\ks<0$. By Corollary \ref{Cor:key_est} with $d=1$, we have 
\begin{align*}
&\left|\mathcal{F}_g(\lambda\xi,\lambda\epsilon) -\lambda^{4d+4}\mu^2 \mathcal{F}\ks \right|\\
&\leq C\mu^{\frac{2(n-1)}{n-2}}\lambda^{\frac{(4d+4)(n-1)}{n-2}}+C\mu\lambda^{2d+2+\frac{n-2}{2}}\rho^{\frac{2-n}{2}}+C\lambda^{n-2}\rho^{2-n}
\end{align*}
for all $\ks\in\Omega$, equivalently,
\begin{align*}
&\left|\lambda^{-4d-4}\mu^{-2}\mathcal{F}_g(\lambda\xi,\lambda\epsilon) - \mathcal{F}\ks \right|\\
\leq& C \mu^{\frac{2}{n-2}}\lambda^{\frac{4d+4}{n-2}}+C\mu^{-1}\lambda^{\frac{n-4d-6}{2}}\rho^{\frac{2-n}{2}}+C\mu^{-2}\lambda^{n-4d-6}\rho^{2-n}
\end{align*}
for all $\ks\in \Omega$. If $\mu^{-2}\lambda^{n-4d-6}\rho^{2-n}$ is sufficiently small, then we have 
\begin{align*}
\mathcal{F}_g (0,\lambda) <\inf_{\ks\in\d\Omega'}\mathcal{F}_g(\lambda\xi,\lambda\epsilon)<0.
\end{align*}
Consequently, there exists $(\bar{\xi}, \bar{\epsilon})\in \Omega'$ such that 
\begin{align*}
\mathcal{F}_g (\lambda \bar{\xi},\lambda \bar{\epsilon})=\inf_{\ks\in\Omega'}\mathcal{F}_g(\lambda\xi,\lambda\epsilon)<0.
\end{align*}
It follows from Proposition \ref{prop:v_critical_pt_F} that the function $v=v_{(\lambda\bar{\xi}, \lambda\bar{\epsilon})}$ obtained in Proposition \ref{prop:sol_v} is a positive smooth solution to \eqref{PDE:single_v}.
By definition of $\mathcal{F}_g$ we have 
\begin{align*}
&2(n-2)\int_{\Rn}v^{\frac{2n}{n-2}}-\frac{n-2}{n-1}T_c \int_{\d\Rn}v^{\frac{2(n-1)}{n-2}}\\
=&\mathcal{F}_g(\lambda\bar{\xi}, \lambda\bar{\epsilon})+2(n-2)\int_{\Rn}u_{(\lambda\bar{\xi}, \lambda\bar{\epsilon})}^{\frac{2n}{n-2}}-\frac{n-2}{n-1}T_c\int_{\d\Rn}u_{(\lambda\bar{\xi}, \lambda\bar{\epsilon})}^{\frac{2(n-1)}{n-2}},
\end{align*}
whence
\begin{align*}
&2(n-2)\int_{\Rn}v^{\frac{2n}{n-2}}-\frac{n-2}{n-1}T_c \int_{\d\Rn}v^{\frac{2(n-1)}{n-2}}\\
<&2(n-2)\int_{\Rn}u_{(0, 1)}^{\frac{2n}{n-2}}-\frac{n-2}{n-1}T_c\int_{\d\Rn}u_{(0, 1)}^{\frac{2(n-1)}{n-2}}.
\end{align*}
By \eqref{annihilator esti} and Proposition \ref{prop:conf_operators} we estimate
\begin{align*}
\|v-u_{(\lambda\bar{\xi}, \lambda\bar{\epsilon})}\|_{L^{\frac{2n}{n-2}}(B_{\lambda}^+(0))}\leq \|v-u_{(\lambda\bar{\xi}, \lambda\bar{\epsilon})}\|_{L^{\frac{2n}{n-2}}(\d\Rn)}\leq C\alpha.
\end{align*}
Then, 
\begin{align*}
|B_{\lambda}^+(0)|^{\frac{n-2}{2n}}\displaystyle{\sup_{B_{\lambda}^+(0)}v}\geq \|v \|_{L^{\frac{2n}{(n-2)}}(B_\lambda^+ (0))} \geq -C\alpha +\|u_{(\lambda\bar{\xi}, \lambda\bar{\epsilon})}\|_{L^{\frac{2(n-1)}{(n-2)}}(B_\lambda^+ (0))}.
\end{align*}
Hence, if $\alpha$ is sufficiently small, then we obtain
\begin{align*}
\displaystyle{\sup_{B_{\lambda}^+(0)}} v\geq C \lambda^{\frac{2-n}{2}}.
\end{align*}
This completes the proof.
\end{proof}
\begin{theorem}\label{Thm: half_space}
Let $n\geq 62$, then there exists a smooth Riemannian metric $g$ on $\overline \Rn$ with the following properties:
\begin{enumerate}
\item[(a)] $g_{ab}(x)=\delta_{ab}$  for $\overline \Rn \setminus B_{1/2}^+(0)$;
\item[(b)] $g$ is not conformally flat;
\item[(c)] $\d\Rn$ is totally geodesic with respect to the induced metric of $g$;
\item[(d)] there exists a sequence of positive smooth functions $\{v_{\nu}; \nu \in \mathbb{N}\} $ satisfying 
\begin{align*}
\begin{cases}
\displaystyle \Delta_g v_{\nu} -c_n R_g v_{\nu}+n(n-2)v_{\nu}^{\frac{n+2}{n-2}}=0,&\quad \mathrm{in~~} \Rn,\\
\displaystyle \frac{\d v_{\nu}}{\d x_n}=(n-2)T_c v_{\nu}^{
\frac{n}{n-2}}, &\quad\mathrm{on~~} \d\Rn,
\end{cases}
\end{align*}
for all $\nu$. Moreover, there hold
\begin{align*}
&2(n-2)\int_{\Rn}v_{\nu}^{\frac{2n}{n-2}}-\frac{n-2}{n-1}T_c \int_{\d\Rn}v_{\nu}^{\frac{2(n-1)}{n-2}}\\
<&2(n-2)\int_{\Rn}u_{(0, 1)}^{\frac{2n}{n-2}}-\frac{n-2}{n-1}T_c\int_{\d\Rn}u_{(0, 1)}^{\frac{2(n-1)}{n-2}}
\end{align*}
for all $\nu$, i.e. $I_{(\Rn,g)}[v_\nu]<S_c$, and $\displaystyle{\sup_{B_1^+ (0)}}v_{\nu} \to \infty $ as $\nu \to \infty$.
\end{enumerate}
\end{theorem}
\begin{proof}
Let $\chi$ be a smooth cut-off function in $\mathbb{R}$ such that $0\leq \chi(t)\leq 1$ for $t \in \mathbb{R}$, $\chi(t)=1$ for $t\leq 1$ and $\chi(t)=0$ for $t\geq 2$. We define a trace-free symmetric two-tensor in $\overline{\Rn}$ by 
\begin{align*}
h_{ab}(x)=\displaystyle{\sum_{N=N_0}^{\infty}}\chi (4N^2 |x-x_N|)2^{-N}f(2^N |x'-x_N|^2)H_{ab}(x-x_N),
\end{align*}
where $x_N=(\frac{1}{N},0,\cdots,0)\in\d\Rn$. Observe that $h$ is smooth and satisfies $h_{an}(x)=0$ and $\d_n h_{ab}(x)=0$ on $\d\Rn$. We choose $\alpha>0$ to be the constant in Proposition \ref{main1} and $N_0$ sufficiently large, then $h_{ab}(x)=0$ for $|x|\geq \frac{1}{2}$ and $|h(x)|+|\d h(x)| +|\d^2 h(x)| \leq \alpha$ for all $x\in \overline \Rn$. Thus, the desired assertion follows from Proposition \ref{main1} with $\lambda=2^{-N/2}, \rho=(2N)^{-2}, \mu=2^{-N}$.
\end{proof}

\end{document}